\documentclass{article}
\usepackage[utf8]{inputenc}
\usepackage[T1]{fontenc}
\usepackage{eliot}

\usepackage[letterpaper, portrait, margin=1.5in]{geometry}

\usepackage{mathabx}

\usepackage{thmtools}

\newcommand{\hull}[1]{hull\left( #1 \right)}
\newcommand{\core}[1]{core\left( #1 \right)}
\theoremstyle{definition}\newtheorem{claim}[theorem]{Claim}

\newcommand{\MCG}[1]{ \mathrm{MCG} ( #1 )}

\newcommand{\Cay}[1]{ \mathrm{Cay} (#1 ) }

\newcommand{\directions}{\mathscr{D}}
\newcommand{\adj}[1]{\bar{#1}} 
\newcommand{\dir}[1]{ \left[ #1 \right]  } 
\newcommand{\length}[1]{\ell \left( #1 \right) } 
\newcommand{\DtoDbar}{ \mathrm{proj}_{\overline{D}} } 
\newcommand{\hulltoDhat}{p} 
\newcommand{\DbartoDhat}{\bar{p}} 
\newcommand{\geod}[2]{\overleftrightarrow{ #1 , #2}} 

\title{Extensions of finitely generated Veech groups}
\author{Eliot Bongiovanni}
\date{13 June 2024}

\begin{document}
\maketitle

\begin{abstract}
Given a closed surface $S$ with finitely generated Veech group $G$ and its $\pi_1(S)$-extension $\Gamma$, there exists a hyperbolic space $\hat{E}$ on which $\Gamma$ acts isometrically and cocompactly.
The space $\hat{E}$ is obtained by collapsing some regions of the surface bundle over the convex hull of the limit set of $G$.
Using the nice action of $\Gamma$ on the hyperbolic space $\hat{E}$, it is shown that $\Gamma$ is hierarchically hyperbolic.
These are generalizations of \cites{ddls-extensions}{ddls-more-extensions}, which assume in addition that $G$ is a lattice.
Because finitely generated Veech groups are among the most basic examples of subgroups of mapping class groups which are expected to qualify as geometrically finite, this result is evidence for the development of a broader theory of geometric finiteness.
\end{abstract}

\tableofcontents

\pagebreak
\section{Introduction}

In the context of Kleinian groups, there is a well-defined notion of ``convex cocompactness'' and a generalization known as ``geometric finiteness''. 
For subgroups of mapping class groups, there is an analogous notion of convex cocompactness, but it is unclear what might be meant by geometric finiteness in this context.
Finitely generated Veech groups are of interest here because 
they are subgroups of the mapping class group (those which stabilize Teichm\"uller disks) and Kleinian groups (since each of those Teichm\"uller disks is isometric to the hyperbolic plane).
In general a finitely generated Veech group is not convex cocompact---neither as a subgroup of a mapping class group nor as a Kleinian group---but it is geometrically finite as a Kleinian group.
It is widely agreed upon that finitely generated Veech groups should qualify as geometrically finite as subgroups of mapping class groups \cite[Section 6]{farb:problems}.
For example, \cite[Theorem 1.4]{tang:undistorted} shows that finitely generated Veech groups are parabolically geometrically finite in the sense of \cite{ddls-more-extensions}, so these serve as a fundamental example for developing a theory of geometric finiteness in the context of mapping class groups.
It is also known that a subgroup of the mapping class group is convex cocompact if and only if its extension group is hyperbolic \cites{fm-convex-cocompactness}{hamenstaedt-hyperbolic-extensions}.
It is suspected that some notion of geometric finiteness of subgroups of mapping class groups corresponds to hierarchical hyperbolicity of the extension group \cites[Problem 6.2]{farb:problems}[Section 1.4]{ddls-more-extensions}[Section 1.1]{russell:multicurve-stabilizers}, and this paper provides further evidence towards that conclusion.

The main result is a generalization of \cite[Theorem 1.1]{ddls-extensions}. The following statement is identical, except that the term ``lattice'' has been generalized to ``finitely generated''.
The $\pi_1(S)$ extension of $G$ is precisely the group $\Gamma$ fitting into the short exact sequence $$1 \rightarrow \pi_1(S) \rightarrow \Gamma \rightarrow G \rightarrow 1, $$ and the vertex subgroups of $\Gamma$ are those that stabilize the vertices of particular Bass-Serre trees upon which $\Gamma$ acts isometrically.

\begin{restatable*}{theorem}{ddlsgeneralization}
    \label{thm:ddls-generalization}
    Suppose $G < \MCG{S}$ is a finitely generated, nonelementary Veech group with extension group $\Gamma$ and let $\Upsilon_1, \dots, \Upsilon_k < \Gamma$ be representatives of the conjugacy classes of vertex subgroups. 
    Then $\Gamma$ admits an isometric action on Gromov hyperbolic space $\hat{E}$, quasi-isometric to the Cayley graph of $\Gamma$ coned off along the cosets of $\Upsilon_1, \dots, \Upsilon_k$. 
\end{restatable*}

While this paper follows the basic construction and argument outline from \cite{ddls-extensions}, the details differ substantially.
First, in both cases the space $\hat{E}$ is constructed from the hyperbolic plane bundle over the convex hull of the limit set of the associated Teichm\"uller disk. However, when $G$ was also assumed to be a lattice, the convex hull of the limit set coincided with the Teichm\"uller disk. 
When $G$ is only finitely generated, the convex hull of the limit set is generally some (strict) subset of the Teichm\"uller disk. 
As a consequence, the proof of the ``fan lemma'' (\cref{lem:fan-lemma}, which generalizes \cite[Lemma 4.12]{ddls-extensions}) is reformulated entirely. 
The fan lemma is the cornerstone for constructing sets that form slim triangles, which are used to prove hyperbolicity of $\hat{E}$ via the guessing geodesics criterion.
(The guessing geodesics criterion essentially says that the existence of \emph{paths} that form slim triangles is sufficient evidence for the existence of \emph{geodesics} that form slim triangles; see \cref{lem:guessing-geodesics}. The version used here is due to Masur-Schleimer \cite{ms:geometry-of-disk-complex} and Bowditch \cite{bowditch:uniform-hyperbolicity}.) 
As in \cite{ddls-extensions}, this paper constructs the ``guessed geodesics'' by concatenating hyperbolic geodesic segments orthogonal to the fibers with saddle connections within fibers.
When $G$ was also assumed to be a lattice, the Veech dichotomy ensured that every saddle connection arose as a boundary component of a cylinder decomposition.
It is therefore necessary to generalize the Veech dichotomy to the case of finitely generated Veech groups.
Though this generalization is known to the experts and follows quickly from other long-established results, it appears to be missing from the literature.
\begin{restatable*}[\theoremttt{Generalized Veech dichotomy}]{theorem}{generalizedveechdichotomy}
    \label{thm:generalized-veech-dichotomy}
    Let $(S,X,q)$ be a flat surface with finitely generated Veech group $G$. Every direction $\alpha \in \Lambda(G) \subset \partial D$ is either minimal and uniquely ergodic or completely periodic and invariant by a parabolic element of the maximal Veech group.
\end{restatable*}
\noindent \cref{sec:veech-dichotomy}, therefore, can be read independently of the rest of the paper.
Its relevance here is to arrive at \cref{cor:to-generalized-veech-dichotomy}, which states that saddle connections either arise as boundary components of cylinder decompositions (as in the classical Veech dichotomy) or are associated to directions not in the limit set of $G$.
Therefore when $G$ is only finitely generated, it is possible for the ``guessed geodesics'' to include saddle connections which are arbitrarily long---an issue not encountered when $G$ was also assumed to be a lattice---and this complicates the process of proving that these paths form slim triangles.

It is important to note that some arguments in this paper require that the finitely generated Veech group $G$ has at least one parabolic element.
If $G$ has no parabolic elements (equivalently, if $G$ is convex cocompact) then it is already known that the extension group is hyperbolic by the aforementioned results of \cites{fm-convex-cocompactness}{hamenstaedt-hyperbolic-extensions}.
However, the arguments in this paper provide a new proof in this special case with little extra work.
\begin{restatable*}[{special case of \cites{fm-convex-cocompactness}{hamenstaedt-hyperbolic-extensions}}]{theorem}{oldtheoremnewproof}
    \label{thm:new-old}
    Let $G < \MCG{S}$ be a finitely generated Veech group with extension group $\Gamma$. If $G$ has no parabolic elements, then $\Gamma$ is Gromov hyperbolic.
\end{restatable*}

After establishing the main result of \cref{thm:ddls-generalization}, hierarchical hyperbolicity of the extension group is a natural next step.
\begin{restatable*}[\theoremttt{Hierarchical hyperbolicity}]{theorem}{hierarchicalhyperbolicity}
    \label{thm:hhs}
    Let $G < \MCG{S}$ be a finitely generated Veech group with extension group $\Gamma$. Then $\Gamma$ is a hierarchically hyperbolic group.
\end{restatable*}
\noindent The proof of this theorem proceeds almost exactly as in \cite{ddls-more-extensions}, after patching one subcase.

A paper currently in progress will prove that an extension of a finitely generated Veech group is quasi-isometrically rigid.

\subsection*{Outline}
\cref{sec:basics} surveys the background information needed to approach this problem and refers to resources that explore these concepts in more depth.
\cref{sec:veech-dichotomy} is dedicated to a generalization of the classical Veech dichotomy, which lays the foundation for the rest of the paper.
\cref{sec:construction} constructs the space $\hat{E}$ featured in \cref{thm:ddls-generalization}.
\cref{sec:hyperbolicity} proves that $\hat{E}$ is hyperbolic (\cref{thm:hat-e-is-hyperbolic}), 
and shows how the rest of the statement of \cref{thm:ddls-generalization} follows.
A key result for the proof of hyperbolicity of $\hat{E}$, \cref{thm:collapsed-preferred-paths-form-slim-triangles}, is deferred to and comprises the entirety of \cref{sec:slimness}.
Finally, \cref{sec:hhs} proves that an extension of a finitely generated Veech group is hierarchically hyperbolic.
For the reader's convenience, a list of symbols and references to the pages on which they first appear is included at the end.

\subsection*{Acknowledgements}
I would like to thank my PhD advisor, Chris Leininger, for being extraordinarily generous with his time and attention as I prepared this paper. Chris provided valuable insights into existing research and ensured that the final manuscript is as sensible as possible, though neither of us could come up with a less silly term for the ``horopoints'' described in \cref{def:horoballs} and appearing throughout the paper.

I am financially supported in part by NSF-1842494 and NSF DMS-1745670.

\section{Background Concepts}
\label{sec:basics}

This section surveys concepts used throughout the paper. References for deeper reading are included at the beginning of each subsection.

\subsection{Fundamental geometry}

\subsubsection{Paths in metric spaces}
See \cite[Chapter I.1]{bh-nonpositive-curvature}.

Let $(X,d)$ be a metric space. A \termttt{path} from $x$ to $y$ in $X$ is a continuous map $c$ from an interval $[0,\ell] \subset \R$ to $X$ with $c(0) = x$ and $c(\ell)=y$.
The \termttt{length} of the path $c$ is 
$$ \sup_P \sum_{i=1}^{n_P} d(c(t_{i-1}), c(t_{i})), $$
where the supremum is taken over all partitions $P$ of $[0, \ell]$ with $t_0 = 0$ and $t_{n_P} = \ell$.
If the length of $c$ is finite, then $c$ is called \termttt{rectifiable}.
The space $X$ is called \termttt{rectifiably path connected} if any two points are connected by a rectifiable path.
If for any $x,y \in X$ the distance $d(x,y)$ is precisely the infimum of the lengths of all paths from $x$ to $y$, then $(X,d)$ is called a \termttt{length space}.

A path $c$ is called a \termttt{geodesic} if for all $t,t' \in [0, \ell]$,
$$ d \left( c(t),c(t') \right) = \lvert t - t' \rvert. $$
The image of the map $c$ is also sometimes referred to as a geodesic.
The space $X$ is called a \termttt{geodesic space} if any two points are connected by a geodesic.
A geodesic between $x$ and $y$ in a geodesic space is typically denoted $[x,y]$, which is the notation used throughout this section.
    \nomenclature[(x,y)]{$[x,y]$}{geodesic segment between $x$ and $y$}

\subsubsection{Quotient spaces}
See \cite[2, 64-70]{bh-nonpositive-curvature}.

Let $(X,d)$ be a metric space, and let $\sim$ be an equivalence relation on elements of $X$. 
A \termttt{chain} from $x$ to $y$ in $X$ is a sequence $\set{x_1,y_1, x_2, y_2, \dots, x_n, y_n}$ with $y_i \sim x_{i+1}$.
For $x',y' \in X' = X/\sim$ representing $x,y \in X$, respectively, the \termttt{quotient pseudometric} is defined by
$$ d'(x',y') := \inf_{C} \sum_{i=1}^n d(x_i,y_i), $$
where the infimum is taken over all chains joining $x$ to $y$.
A \termttt{pseudometric} meets all of the requirements to be a metric except that it is perhaps not \termttt{positive-definite}, meaning that $d'(x',y') = 0$ need not imply $x' = y'$.
The following lemma is used several times in this paper to verify that a quotient space is in fact a length space.

\begin{lemma}[{\cite[Lemma I.5.20]{bh-nonpositive-curvature}}]
\label{lem:bh-length-space}
Let $(X,d)$ be a length space, let $\sim$ be an equivalence relation on $X$ and let $d'$ be the quotient pseudometric on $X' = X / \sim$. If $d'$ is a metric then $(X', d')$ is a length space.
\end{lemma}

\subsubsection{Gromov hyperbolicity}
\label{sec:gromov-hyperbolicity}
See \cite[Chapter III.H.1]{bh-nonpositive-curvature} and \cite{vaisala}.

Let $(X,d)$ be a metric space with $x \in X$. The \termttt{Gromov product of $y,z \in X$ with respect to $x$} is
$$ \left( y \cdot z \right)_x := \frac{1}{2} \left( d(y,x) + d(z,x) - d(y,z) \right). $$
For $\delta \geq 0$, $X$ is called \termttt{$\delta$-hyperbolic} (or \termttt{Gromov hyperbolic with hyperbolicity constant $\delta$}) if 
$$ (x \cdot y)_w \geq \min\left\{(x \cdot z)_w, (y \cdot z)_w\right\} - \delta $$
for all $w,x,y,z \in X$.
When $X$ is a geodesic space, $\delta$-hyperbolicity is equivalent to every geodesic triangle in $X$ being $\delta'$-slim for some $\delta' \geq 0$ depending on $\delta$:
A \termttt{geodesic triangle} consists of three points $x,y,z \in X$ and choices of geodesics $[x,y]$, $[y,z]$, and $[x,z]$
and is \termttt{$\delta$-slim} if each side is contained in the $\delta$-neighborhood of the other two, i.e.,
\begin{align*}
    &[x,z] \subset N_\delta \left( [x,y] \cup [y,z] \right), \\
    &[x,y] \subset N_\delta \left( [y,z] \cup [x,z] \right), \text{ and} \\
    &[y,z] \subset N_\delta \left( [x,z] \cup [x,y] \right), \\
\end{align*}
where $N_\delta$ denotes the $\delta$-neighborhood.

In this paper, any mention of hyperbolicity is Gromov hyperbolicity (with some hyperbolicity constant).

\subsubsection{Coarse geometry}
See \cite[138-144]{bh-nonpositive-curvature} and \cite{vaisala}.

Let $(X,d_X)$ and $(Y,d_Y)$ be metric spaces. Given $\lambda \geq 1$ and $\varepsilon \geq 0$, a map $f: X \rightarrow Y$ is called a \termttt{$(\lambda, \varepsilon)$-quasi-isometric embedding} if for all $x,y \in X$,
$$ \frac{1}{\lambda} d_X(x,y) - \varepsilon \leq d_Y(f(x),f(y)) \leq \lambda d_X(x,y) + \varepsilon. $$
When the domain $X$ is an interval in $\R$ or $\Z$, the map $f$ is called a \termttt{quasi-geodesic}; the image of the map is also sometimes referred to as a quasi-geodesic.
If in addition there is a constant $K \geq 0$ so that $N_K \left( f(X) \right) = Y$---that is, if $f$ is \termttt{coarsely surjective}---then $f$ is called a \termttt{quasi-isometry}, and the spaces $X$ and $Y$ are said to be quasi-isometric.
If $f$ is a quasi-isometry, then there exists a \termttt{quasi-inverse} of $f$, which is a quasi-isometry $g: Y \rightarrow X$ so that $d_X(x,g(f(x)))$ and $d_Y(y,f(g(y)))$ are uniformly bounded for all $x \in X$ and $y \in Y$.

Gromov hyperbolicity of length spaces is an invariant of quasi-isometry: If $(X,d_X)$ and $(Y,d_Y)$ are length spaces, $X$ is hyperbolic, and $f: X \rightarrow Y$ is a quasi-isometry, then $Y$ is hyperbolic (perhaps with a different hyperbolicity constant from that of $X$) \cite[16]{vaisala}. 

\subsubsection{Cayley graphs}

Given a finitely generated group $G$ with generating set $\calA$, the \termttt{Cayley graph of $G$}, denoted $\Cay{G}$, is the graph whose vertices are elements of $G$ and whose edges connect $g$ to $ga$ for any $g \in G$, $a \in \calA$. Defining all edges to be unit length makes the Cayley graph a metric space, where the induced metric is precisely the word metric with respect to the set $\calA$.
The group $G$ equipped with the word metric is quasi-isometric to its Cayley graph, and any two Cayley graphs for $G$ (obtained from two different generating sets) are quasi-isometric (see \cite[139-141]{bh-nonpositive-curvature}).

Given a finite family of subgroups $\calH = \{H_1, H_2, \dots, H_n\}$ of $G$, the \termttt{Cayley graph of $G$ with respect to $\calH$} or \termttt{coned-off Cayley graph}, denoted $\Cay{G, \calH}$ is the graph obtained from $\Cay{G}$ by adding a vertex $V_{gH}$ for each left coset $gH$ (with $g \in G$ and $H \in \calH$) and attaching $V_{gH}$ by an edge of length $1/2$ to each vertex of $\Cay{G}$ which is an element of the coset $gH$. The isometric left action of $G$ on $\Cay{G}$ extends to an isometric action of $G$ on $\Cay{G, \calH}$: For all $g,g' \in G$ and $H \in \calH$, set $g V_{g'H} = V_{gg'H}$. 
The nontrivial vertex stabilizers of the action are conjugate to subgroups of $\calH$.
See \cite{cc-relative-hyperbolicity}, including the following variant of the Schwarz-Milnor lemma.

\begin{theorem}[{\cite[Theorem 5.1]{cc-relative-hyperbolicity}}]
\label{thm:groupy-schwarz-milnor}
Let $G$ be a finitely generated group and suppose that $G$ admits a discontinuous (that is, with discrete orbits), cocompact, isometric action on a length space $X$. Let $\calH$ denote a collection of subgroups of $G$ consisting of exactly one representative of each conjugacy class of maximal isotropy subgroups for the action of $G$ on $X$. Then $\calH$ is finite and, for any finite generating set $\calA$ of $G$, the coned-off Cayley graph $\Cay{G, \calH}$ is quasi-isometric to $X$. In particular, if $X$ is a hyperbolic space then the coned-off Cayley graph is hyperbolic. 
\end{theorem}

\subsubsection{Mapping class groups}
See \cite{fm-primer}.

Let $S$ be a closed, orientable surface of genus at least $2$.
The \termttt{mapping class group of $S$} is
$$ \MCG{S} := \pi_0 \left( \mathrm{Homeo}^+(S) \right). $$
That is, the mapping class group is the group of isotopy classes of orientation-preserving homeomorphisms on $S$.

Mark a point $*$ on $S$ and denote the marked surface by $\dot{S}$. The mapping class group of the marked surface is then
$$ \MCG{\dot{S}} := \pi_0 \left( \mathrm{Homeo}^+(S, *) \right), $$
i.e., the isotopy classes of elements of $\mathrm{Homeo}^+(S)$ which fix the marked point $*$.
These mapping class groups fit into the \termttt{Birman exact sequence}, 
\begin{equation}
\label{eq:birman-exact-sequence}
    1 \rightarrow \pi_1(S) \rightarrow \MCG{\dot{S}} \rightarrow \MCG{S} \rightarrow 1,
\end{equation}
where $\MCG{\dot{S}} \rightarrow \MCG{S}$ is the map which forgets the marked point on $S$ \cites{fm-primer}{birman-book}.
For a subgroup $G < \MCG{S}$, denote its preimage by $\Gamma_G < \MCG{\dot{S}}$ and observe that it fits into a short exact sequence
$$ 1 \rightarrow \pi_1(S) \rightarrow \Gamma_G \rightarrow G \rightarrow 1 $$
which includes into the Birman exact sequence above.
The group $\Gamma_G$ is called the \termttt{$\pi_1(S)$-extension of $G$} (or simply the extension of $G$), and it is the fundamental group of an $S$-bundle with monodromy an isomorphism onto $G$.

\subsection{Flat surfaces}
See \cites{gardiner:teichmueller-theory-and-quadratic-differentials}{strebel:quadratic-differentials}. 

Let $S$ be a closed, connected, and oriented surface of genus at least $2$. 
Equip $S$ with a \termttt{complex structure} $X$, which is an atlas of charts $\set{z_\alpha: U_\alpha \rightarrow \C}$ whose transition functions $z_\beta^{-1} \circ z_\alpha$ are biholomorphic wherever the composition is defined. 
A \termttt{quadratic differential} for $X$, denoted $q_X$ or more simply $q$, is a nonzero holomorphic section of the square of the canonical line bundle over $(S,X)$. 
\nomenclature[q]{$q$}{quadratic differential or flat metric for a complex structure in the Teichm\"uller space of a closed surface of genus at least $2$}
Note in particular that any nonzero quadratic differential has finitely many zeroes. 
The pair $(X,q)$ are a \termttt{flat structure} on $S$, described as follows.

In a small disk neighborhood of a nonzero point $p$, choose a coordinate chart $z$ so that $p$ corresponds to $z(p) = 0$ and pick a branch of $q^{1/2}(z)$. The \termttt{natural coordinate} or \termttt{preferred coordinate} in a neighborhood of $p$ is given by
    $$ \zeta(z) = \int_0^z q^{1/2}(u) du. $$
In this coordinate, $q$ is given by $q(z) dz^2 = d\zeta^2$. In the neighborhood of a zero of order $k \geq 1$ there are natural coordinates such that $q(z) dz^2 = \zeta^k d\zeta^2$. 

Away from the zeroes of $q$, the transition functions for overlapping preferred coordinates of $q$ are locally given by $z \mapsto \pm z + c$ for some $c \in \mathbb{C}$.
Because the Euclidean metric is invariant under these transition functions, the Euclidean metric pulls back to a metric on $S$ minus the zeroes of $q$. 
To complete the pullback metric, each zero of order $k-2$ is filled back in so that a neighborhood of the point is isometric to the image of $k$ Euclidean half planes glued together in a cyclic pattern by identifying the positive real axis of one half plane with the negative real axis of another (and only one other) half plane. 
A point of $S$ corresponding to a zero of $q$---that is, a singular point---of order $k-2$ is called a \termttt{cone point} with \termttt{cone angle} $k \pi$. The resulting metric is called a \termttt{flat metric}, which is also denoted by $q$. 
(Using the same notation for both the quadratic differential and the flat metric is a bit imprecise: The flat metric determines the quadratic differential up to multiplication by a nonzero complex number.)

Given a complex structure $X$ on $S$ and an associated flat metric $q$, the triple $(S,X,q)$ is a \termttt{flat surface}. 
    \nomenclature[(S,X,q)]{$(S,X,q)$}{flat surface determined by a closed surface $S$ of genus at least $2$, a choice of complex structure $X$, and its associated quadratic differential $q$}
Where the structure and metric are implied, a flat surface is often denoted simply as $S$.

\subsubsection{Directions}
\label{sec:directions}

Because the transition functions for overlapping preferred coordinates for $q$ are locally given by $z \mapsto \pm z + c$, a line in the tangent space at any nonsingular point can be parallel translated (everywhere except the cone points) to produce a smooth line field, which corresponds to a line in the projective tangent space at any nonsingular point.
In other words, any tangent line to the surface has a distinguishable direction, up to a rotation by $\pi$, which is consistent everywhere on the surface away from the cone points.
This is known as the \termttt{space of directions}. It is sometimes denoted by $\P^1(q)$, but in this paper the notation $\partial D$ (introduced in \cref{sec:teichmueller-space}) is used instead.

\subsubsection{Geodesics}

Because the flat metric $q$ is Euclidean away from the cone points, (local) geodesics on $S$ minus the cone points are straight lines. 
A geodesic containing a cone point locally consists of two straight line segments meeting at the cone point and forming angles of at least $\pi$ on both sides.
A geodesic between two cone points and with no cone points on its interior is called a \termttt{saddle connection}.
In particular, a saddle connection $\sigma$ determines a line in the tangent space at any of its interior points, and therefore determines a point in the space of directions which is denoted by $\dir{\sigma}$.
The length of a saddle connection $\sigma$ is denoted by $\length{\sigma}$.
    \nomenclature[l(sigma)]{$\length{}$}{length of a saddle connection $\sigma$}

\subsubsection{Foliations}

The line field obtained by parallel translating a tangent line around $S$ minus the cone points also corresponds to a foliation of $S$ minus the cone points by geodesics, which extends to a singular foliation over all of $S$ when the cone points are included back in.
So for any direction $\alpha$ there is a corresponding (singular) foliation $\calF(\alpha)$ in direction $\alpha$.

It is possible that for particular $\alpha$ the foliation $\calF(\alpha)$ defines a \termttt{cylinder decomposition} in which every nonsingular leaf is a closed geodesic, and the singular leaves are concatenations of saddle connections which separate $S$ into a union of Euclidean cylinders.
A description of some of the directions $\alpha$ for which this occurs is given by \cref{thm:generalized-veech-dichotomy}.

\subsubsection{The universal cover}

Let $\tilde{S}$ denote the universal cover of $S$. 
The complex structure $X$ and quadratic differential $q$ on $S$ can be pulled back to $\tilde{S}$, and the covering also gives a canonical identification of the directions on $S$ with those on $\tilde{S}$.
The covering map from $\tilde{S}$ to $S$ sends cone points to cone points (and saddle connections to saddle connections) and is a local isometry.
For simplicity, the same notations are used for the structure, quadratic differential (and associated metric), and space of directions for $\tilde{S}$.

Because every cone point on $S$ has cone angle greater than $2\pi$, both $S$ and $\tilde{S}$ are nonpositively curved via Gromov's link condition (see for instance \cite[Chapter II.5]{bh-nonpositive-curvature}).
Therefore the pulled back metric on $\tilde{S}$ is $\CAT{0}$, and so $\tilde{S}$ is uniquely geodesic.
The geodesics are analogous to those in $S$, consisting of concatenations of saddle connections (with perhaps a straight line segment at the beginning or end of the path).
Given a Euclidean cylinder in $S$, the preimage is a union of strips, and the covering map restricts to a universal covering of the cylinder on each strip.

\subsection{The Teichm\"uller space}
\label{sec:teichmueller-space}

See \cites{masur:teichmuller-space}[Section 8.2]{gl:quasiconformal-teichmueller-theory}.

Two complex structures $X$ and $Y$ on $S$ are called \termttt{equivalent} if there is a map $f: (S,X) \rightarrow (S,Y)$, biholomorphic in the coordinate charts, which is isotopic to the identity on $S$. 
The \termttt{Teichm\"uller space} of $S$, denoted $\calT(S)$, is the space of equivalence classes of complex structures on $S$. The notation $X$ is used both for a particular complex structure and its isotopy class $X \in \calT(S)$. The Teichm\"uller space of $S$ comes equipped with a metric known as the Teichm\"uller metric, which will not appear in this paper explicitly, but whose relevant features are described below.

Denoting by $\dot{S}$ the surface $S$ with a marked point, the space $\calT(\dot{S})$ is the space of isotopy classes of complex structures in which isotopies are also required to fix the marked point.
A fibration of Teichm\"uller spaces called the \termttt{Bers fibration} is given by 
\begin{equation}
\label{eqn:bers-fibration}
    \tilde{S} \rightarrow \calT(\dot{S}) \rightarrow \calT(S),
\end{equation}
obtained by forgetting the marked point, where the fiber over a point $X \in \calT(S)$ is canonically identified with $\tilde{S}$ \cites{bers-fiber-spaces}{ls-hyperbolic-spaces-in-teichmueller-spaces}.

\subsubsection{The Teichm\"uller disk}

Given a flat surface $(S,X,q)$, where $q$ has preferred coordinates $\zeta_i$, a new complex structure can be obtained from any $A \in \mathrm{SL}_2(\R)$ acting on by applying $A$ to the given atlas---that is, a new atlas $\set{A \circ \zeta_i}$, where $A$ is acting as a linear transformation of $\R^2 \cong \C$. 
The new complex structure is denoted $A \cdot (X,q) = (A \cdot X, A \cdot q)$. 
Note that this deformation preserves the zeroes (including their orders) of the original structure $(X,q)$. 
The map $\mathrm{SO}(2) A \mapsto A \cdot X$ gives a homeomorphism from $\mathrm{SO}(2) \backslash \mathrm{SL}_2(\R)$ to the image of the orbit $D \subset \calT(S)$ of $(X,q)$, since $\mathrm{SO}(2)$ preserves the underlying complex structure. The disk $D$ is called the \termttt{Teichm\"uller disk} of $q$, on which the Teichm\"uller metric is the push-forward of the Poincar\'e metric (by Teichm\"uller's theorem). 
\nomenclature[D]{$D$}{Teichm\"uller disk (determined by a flat structure $(X,q)$, although this is often dropped from the notation)}
As a consequence, 
$$ \mathrm{Isom}^+(D) \cong \mathrm{PSL}_2(\R), $$
where the latter is the group of orientation-preserving isometries of the hyperbolic plane.
It is often convenient to think of $D$ as the Poincar\'e disk model of the hyperbolic plane, where each point represents a complex structure on $S$ and traversing a geodesic between two points in $D$ corresponds to varying the underlying flat structure by affine deformations.

More specifically, a geodesic through $X \in D$ is the map $t \mapsto A_t \cdot (X,q)$, where $\{ A_t \}_{t \in \R}$ is a symmetric, 1-parameter hyperbolic subgroup of $\mathrm{SL}_2(\R)$---that is, given by a matrix conjugate to 
$$ \pm \begin{bmatrix}
    e^t & 0 \\
    0 & e^{-t}
\end{bmatrix}$$
by an element of $\mathrm{SO}(2)$.
All geodesics in $D$ can be obtained by $\mathrm{SL}_2(\R)$ conjugates of the same 1-parameter family.
The unit eigenvectors of such a transformation are orthogonal, leading to an identification of the boundary circle $\partial D$ with the space of directions $\P^1(q)$ by associating the endpoint of the positive ray with the direction of the contracting eigenvector. 
As the ray approaches the boundary point associated to direction $\alpha$, the length of any saddle connection in direction $\alpha$ (if one exists) shrinks exponentially. Since this occurs for a ray based at any point and ending at the boundary point associated to direction $\alpha$, any horocycle based at the boundary point is a level set for the length of a saddle connection in direction $\alpha$. 

\subsection{Veech groups}
See \cites{thurston-diffeomorphisms-of-surfaces}[Chapters 7-8]{gl:quasiconformal-teichmueller-theory}{hs-veech-surfaces}.

For a flat surface $(S,X,q)$ and its associated Teichm\"uller disk $D \subset \calT(S)$, a \termttt{Veech group} of $q$ is a subgroup of the stabilizer of $D$.
(Note also that the stabilizer of $D$ is a subgroup of the mapping class group of $S$.)
In this paper, a choice of Veech group of $q$ is denoted by $G$.
Equivalently, a Veech group is a subgroup of the affine homeomorphisms of $(S,X,q)$ in preferred coordinates for $q$ which fix the cone points, projected into the mapping class group.

Any affine homeomorphism in the Veech group has a derivative in preferred coordinates which is well-defined up to sign and determines an element of $\mathrm{PSL}_2(\R)$, the group of orientation-preserving isometries of $D$---or, equivalently, of the hyperbolic plane. 
The action of the Veech group $G$ on $D$ is conjugate to an action on $\H$ via the derivative (an element of $\mathrm{PSL}_2(\R)$).
Since these two actions are essentially the same, both perspectives are employed throughout this paper.
Elements of $G$ are referred to as \termttt{parabolic}, \termttt{hyperbolic}, or \termttt{elliptic} according to whether their images under the derivative homomorphism are parabolic, hyperbolic, or elliptic isometries of $\H$, respectively. 
Finally, because $\MCG{S}$ acts properly discontinuously on $\calT(S)$, a Veech group $G$ acts properly discontinuously on $\H$.
Therefore the image of $G$ under the derivative homomorphism is a \termttt{Fuchsian group}, that is, a discrete subgroup of $\mathrm{PSL}_2(\R)$.

In this paper, all Veech groups are assumed to be finitely generated.

\subsubsection{Limit points}

The \termttt{limit set} of $G$, denoted $\Lambda(G)$, is the set of all possible limit points of a $G$-orbit, $G z$, for $z \in D$. 
A Fuchsian group $G$ is called \termttt{nonelementary} if $\abs{\Lambda(G)} > 2$, which implies that $G$ is not virtually cyclic.
The limit set $\Lambda(G)$ is contained in the boundary circle $\partial D$ \cite[Corollary 2.2.7]{katok-fuchsian-groups}. 
By identifying the boundary circle $\partial D$ with $\P^1(q)$, each $\alpha \in \Lambda(G)$ corresponds to a direction $\alpha \in \P^1(q)$.
Throughout this paper, limit points and their associated directions are referred to interchangeably.

A limit point is called a \termttt{parabolic fixed point} if it is the fixed point of a parabolic element of $G$.
A limit point $\alpha \in \partial D$ is a \termttt{conical limit point} if for any ray $R$ ending at $\alpha$ there is a point $Y \in D$, a sequence $\set{g_i}_{i=1}^\infty$ of elements of $G$, and an $\epsilon > 0$ such that $\set{g_i Y}_{i=1}^\infty$ converges to $\alpha$ within the $\epsilon$-neighborhood of $R$ in $D$ {\cites[617]{ratcliffe-foundations}}. 
Since a finitely generated Veech group is Fuchsian, the following characterization of limit points applies.

\begin{theorem}[{\cite[Theorem 10.2.5]{beardon-discrete-groups}}]
\label{thm:limit-pts-of-fin-gen-fuchsian-gps}
A Fuchsian group is finitely generated if and only if each limit point is either a parabolic fixed point or a conical limit point. 
\end{theorem}

\noindent The precise definition of a conical limit point is only necessary for a proof of the generalized Veech dichotomy in \cref{sec:veech-dichotomy}.

\subsubsection{Convex hull and convex core}
\label{sec:convex-hull-and-convex-core}

The \termttt{convex hull of $G$}, denoted $\hull{G}$, is the intersection of all hyperbolic half planes whose closures in $D \cup \partial D$ contain $\Lambda(G)$.
    \nomenclature[hull G]{$\hull{G}$}{convex hull of $G$}
The convex hull is the minimal closed, convex, $G$-invariant subset of $D$ {\cite[637]{ratcliffe-foundations}}.
Consequently, the \termttt{convex core} of $D / G$, denoted $\core{D / G} := \hull{G} / G$, is the smallest closed convex subset of $D / G$ for which the inclusion is a homotopy equivalence. 
When $G$ is finitely generated, $\abs{\Lambda(G)} \leq 1$ implies that $\core{D / G} = \emptyset$, otherwise $\core{D / G}$ is nonempty {\cite[637]{ratcliffe-foundations}}. 
Moreover, when $G$ is finitely generated, the convex core has finite area, and truncating the convex core along its cusps results in a compact space \cites[Proposition 8.4.3]{thurston-geometry-topology-of-3-manifolds}[Theorem 10.1.2]{beardon-discrete-groups}.

\subsubsection{Lattices}
\label{sec:lattices}

If the quotient space of $D$ under the action of $G$ has finite area, then the Veech group $G$ is called a \termttt{lattice}, and the associated flat surface $(S,X,q)$ is called a \termttt{lattice surface}.
In particular, if $G$ is a lattice then $\hull{G} = D$.

\section{Generalized Veech dichotomy}
\label{sec:veech-dichotomy}

The Veech dichotomy characterizes foliations on a lattice surface.
A generalization to foliations in certain directions on any flat surface follows quickly from several well-known results, but appears to be missing from the literature.
The result most relevant to this paper is \cref{cor:to-generalized-veech-dichotomy} at the end of this section. Otherwise, this section is self-contained.

\begin{definition}
Let $\calF_\alpha$ denote a foliation of $(S,X,q)$ in the direction $\alpha$.
If each of the leaves of $\calF_\alpha$ is dense in $(S,X,q)$, then $\calF_\alpha$ is called \termttt{minimal}; if in addition the transverse measure is unique up to scalar multiplication, then it is called \termttt{uniquely ergodic}.
If, on the other hand, each of of the leaves of $\calF_\alpha$ is either closed or a saddle connection, then $\calF_\alpha$ is called \termttt{completely periodic}.
\end{definition}

A saddle connection is considered a leaf, so when $\calF_\alpha$ is minimal, there are no saddle connections on $S$ in the direction $\alpha$.
When $\calF_\alpha$ is completely periodic, this is means that there is a cylinder decomposition of $S$ in the direction $\alpha$.

\begin{theorem}[\theoremttt{Classical Veech dichotomy}\footnote{This statement is primarily based on \cite[Theorem 5.10]{mt-rational-billiards}.
An alternate statement and proof (in the language of flows rather than foliations) is \cite[Theorem 1]{hs-veech-surfaces}.
Both papers note the proof of \cite[Theorem 3.4]{vorobets-planar-structures} as another resource.
The original theorem statement is due to \cite{veech-teichmueller-curves}.}]
Let $(S,X,q)$ be a lattice surface. 
Every direction $\alpha \in \partial D$ is either minimal and uniquely ergodic or completely periodic and invariant by a parabolic element of the maximal Veech group.
\end{theorem}

When the maximal Veech group $G$ is not a lattice, there may be directions on $S$ which do not correspond to points in the limit set of $G$. The general statement considers only those directions which correspond to limit points for a finitely generated Veech group.
\generalizedveechdichotomy

\begin{remark}
    When the Veech group is a lattice, the limit set is all of $\partial D$. Therefore the classical Veech dichotomy appears as a special case of this theorem.
\end{remark}

\begin{definition}[{\cite[1033]{mt-rational-billiards}}]
    The \termttt{Teichm\"uller geodesic flow} is the one parameter subgroup of $\text{SL}_2(\mathbb{R})$ given by
    $$ g_t = \begin{bmatrix} e^{t} & 0 \\ 0 & e^{-t} \end{bmatrix} $$
    acting on the space of all quadratic differentials.
    A flat structure $(X,q)$ is called \termttt{divergent} if $g_t \cdot (X,q)$ eventually exits every compact set in the moduli space $\calT(S) / \MCG{S}$ as $t \rightarrow \infty$. 
\end{definition}

Similar to the proof of the classical Veech dichotomy, the proof of the generalized Veech dichotomy will require Masur's criterion. 
\begin{theorem}[\theoremttt{Masur's criterion}; {\cite[Theorem 3.8]{mt-rational-billiards}}]
\label{thm:masurs-criterion}
Let $\alpha$ be the vertical direction.
If the foliation $\calF_\alpha$ of $(S,X,q)$ is minimal but not uniquely ergodic, then $(X,q)$ is divergent. 
\end{theorem}

\begin{lemma}[from proof of {\cite[Theorem 1.8]{mt-rational-billiards}}]
\label{rmk:no-saddle-implies-minimal}
If there is no saddle connection in direction $\alpha$, then the foliation $\calF_\alpha$ is minimal. 
\end{lemma}

\begin{lemma}
\label{lem:return-to-compact-implies-no-saddle}
Let $\alpha \in \Lambda(G) \subset \partial D$ and let $R$ be a ray in $D$ ending at $\alpha$. If the image of $R$ under the projection $D \rightarrow D / G$ returns to a compact set infinitely often, then there is no saddle connection on $(S,X,q)$ in direction $\alpha$.
\end{lemma}

\begin{proof}
Note that for all $g \in G$ and $(S,X,q) \in D$, $g \cdot (S,X,q)$ is isometric to $(S,X,q)$---so points of $D/G$ are flat surfaces up to isometry. The length of the shortest saddle connection in direction $\alpha$ is a continuous function on $D / G$, so over any compact subset of $D/G$ there is a lower bound on the length of the shortest saddle connection in direction $\alpha$.

Suppose that there is a saddle connection on $(S,X,q)$ in direction $\alpha$. Without loss of generality assume that $\alpha$ is the vertical direction, so that $R$ is the image of the Teichm\"uller geodesic flow $g_t$. 
Then the length of the shortest saddle connection in the vertical direction tends to zero.
Therefore the image of $R$ in $D/G$ cannot return to a compact set of $D / G$ infinitely often.
\end{proof}

\begin{proof}[Proof of generalized Veech dichotomy]
Because $G$ is finitely generated, $\alpha$ is either a conical limit point or parabolic fixed point (\cref{thm:limit-pts-of-fin-gen-fuchsian-gps}). 
When $\alpha$ is a parabolic direction the foliation in direction $\alpha$ is completely periodic; see for example \cite[Section 6]{thurston-diffeomorphisms-of-surfaces}.

Suppose that $\alpha$ is conical, and let $R$ be a ray based at a point $X \in \hull{G}$ and ending at $\alpha$.
Then there exist a sequence $\set{g_i}_{i=1}^\infty$ of elements of $G$ and an $\epsilon > 0$ large enough that $\set{g_i X}_{i=1}^\infty$ converges to $\alpha$ within the $\epsilon$-neighborhood of $R$ in $D$. 
Then $R$ intersects the closed $\epsilon$-neighborhood of $X$ and its images under each of the isometries $g_i$ (i.e., the translates of the neighborhood around $X$), so $R$ must return to a compact set in $D /G$ infinitely often.
By \cref{lem:return-to-compact-implies-no-saddle} there is no saddle connection in direction $\alpha$, so the foliation $\calF_\alpha$ is minimal (\cref{rmk:no-saddle-implies-minimal}). Finally, because $\calF_\alpha$ is minimal and $X$ is not divergent,
the contrapositive to Masur's criterion (\cref{thm:masurs-criterion}) implies that $\calF_\alpha$ must be uniquely ergodic.
\end{proof}

\begin{corollary}
\label{cor:to-generalized-veech-dichotomy}
For any saddle connection $\sigma$, the associated direction $\dir{\sigma} \in \partial D$ is either parabolic or lies outside of $\Lambda(G)$.
\end{corollary}

\section{Construction}
\label{sec:construction}

Let $S$ be a closed, connected, oriented surface of genus at least 2.
\nomenclature[S]{$S$}{closed, connected, oriented surface of genus at least $2$}
Fix a complex structure $X_0$ on $S$ and a flat metric $q_0$ for $X_0$ so that $(S,X_0,q_0)$ is a flat surface.
    \nomenclature[X0]{$X_0$}{fixed choice of complex structure on $S$}
    \nomenclature[q0]{$q_0$}{flat metric for the complex structure $X_0$}
Let $D$ be the Teichm\"uller disk of $q_0$, let $\rho$ be the Poincar\'e metric on $D$, 
    \nomenclature[rho]{$\rho$}{Poincar\'e metric on the Teichm\"uller disk $D$}
and let $G$ be a finitely generated Veech group of $q_0$. 
    \nomenclature[G]{$G$}{finitely generated Veech group of $q_0$}

See \cite{ddls-extensions} for the motivating construction in the case where $G$ is not only finitely generated but also a lattice.
Their construction of analogous spaces $E$ and $\hat{E}$ differs from the construction given here. 
However, several of their arguments continue to apply in this more general setting essentially verbatim, and their work is cited wherever this is the case.

\subsection{Base spaces}

\subsubsection{Horoballs and horopoints}
\label{sec:horoballs-and-horopoints}

Denote by $\directions$ the set of all directions of all saddle connections on $(S,X_0,q_0)$. Consider $\directions$ to be a subset of $\partial D$ as described in \cref{sec:directions}.
    \nomenclature[D]{$\directions$}{set of all directions of all saddle connections on $(S,X_0,q_0)$, considered as a subset of $\partial D$}
For a saddle connection $\sigma$, denote its direction by $\dir{\sigma} \in \directions$. 
    \nomenclature[(sigma)]{$\dir{\sigma}$}{direction associated to the saddle connection $\sigma$}
The notation $\alpha \in \directions$ is sometimes used to refer to a direction without a specific choice of saddle connection.
    \nomenclature[alpha]{$\alpha$}{an element of $\directions$ (without specific reference to an associated saddle connection)}

Because $G$ is assumed to be finitely generated and not necessarily a lattice, the direction of any saddle connection is associated to either a parabolic limit point or to a point outside the limit set of $G$ (\cref{cor:to-generalized-veech-dichotomy}), which inspires the definitions below.

\begin{definition}
\label{def:horoballs}
Let $\sigma$ be a saddle connection with direction $\dir{\sigma} \in \directions$.
\begin{enumerate}[(i)]
    \item If $\dir{\sigma}$ is a parabolic limit point, then it is called a \termttt{parabolic direction} and $\sigma$ is called a \termttt{parabolic saddle connection}.
    For each such $\dir{\sigma}$, make a choice of closed horoball which is invariant by the maximal parabolic subgroup of $G$ corresponding to $\dir{\sigma}$.
    Choose the horoball to be small enough so that in any fiber over a point in its interior, the length of a saddle connection in direction $\dir{\sigma}$ is no more than one third the length of a saddle connection in any other direction. (This ensures that the horoballs are $1$-separated.)
    Also choose the horoball to be small enough that its distance to $\partial \hull{G}$ is at least one.
    Finally, choose the horoballs so that the set of all horoballs (for all parabolic directions) is $G$-invariant.
    For each parabolic $\dir{\sigma}$, fix this choice of horoball and call it the \termttt{horoball} for $\dir{\sigma}$, denoted $B_{\dir{\sigma}}$.
    \item If $\dir{\sigma}$ is not a parabolic limit point, then it is called a \termttt{nonparabolic direction} and $\sigma$ is called a \termttt{nonparabolic saddle connection}.
    For each such $\dir{\sigma}$, define the \termttt{horopoint} for $\dir{\sigma}$, denoted $B_{\dir{\sigma}}$, to be the $\rho$-closest orthogonal projection of $\dir{\sigma}$ to $\partial \hull{G}$---that is, the unique point $x \in \partial \hull{G}$ such that the (unique, hyperbolic) geodesic through $x$ and $\dir{\sigma}$ meets $\partial \hull{G}$ orthogonally. 
    Note that the set of all horopoints (for all nonparabolic directions) is necessarily $G$-invariant.
\end{enumerate}
    \nomenclature[B alpha]{$B_{\alpha}$}{if $\alpha$ is parabolic, closed horoball in $D$ that is invariant by the maximal parabolic subgroup of $G$ corresponding to $\alpha$; if $\alpha$ is nonparabolic, orthogonal projection of the point $\alpha \in \partial D$ onto $\hull{G}$}
See \cref{fig:horostuff-paths} (blue) for examples. 
The notation $B_*$ is used in both definitions because these objects function analogously later, when it is sometimes helpful to refer to them interchangeably. 
\end{definition}
The set of horopoints is not necessarily pairwise separated by a fixed constant, although by construction the horopoints are separated by a fixed constant from the set of horoballs associated to parabolic saddle connections. 
For each $\alpha \in \directions$ fix a point $X_{\alpha} \in \partial B_{\alpha}$; for nonparabolic directions the only possible choice is $X_{\alpha} = B_{\alpha}$, the horopoint for $\alpha$.
    \nomenclature[X alpha]{$X_{\alpha}$}{choice of fixed point in $\partial B_{\alpha}$; if $\alpha$ is nonparabolic, identical to $B_{\alpha}$}

\begin{figure}[H]
    \centering
    \includegraphics[width=\textwidth]{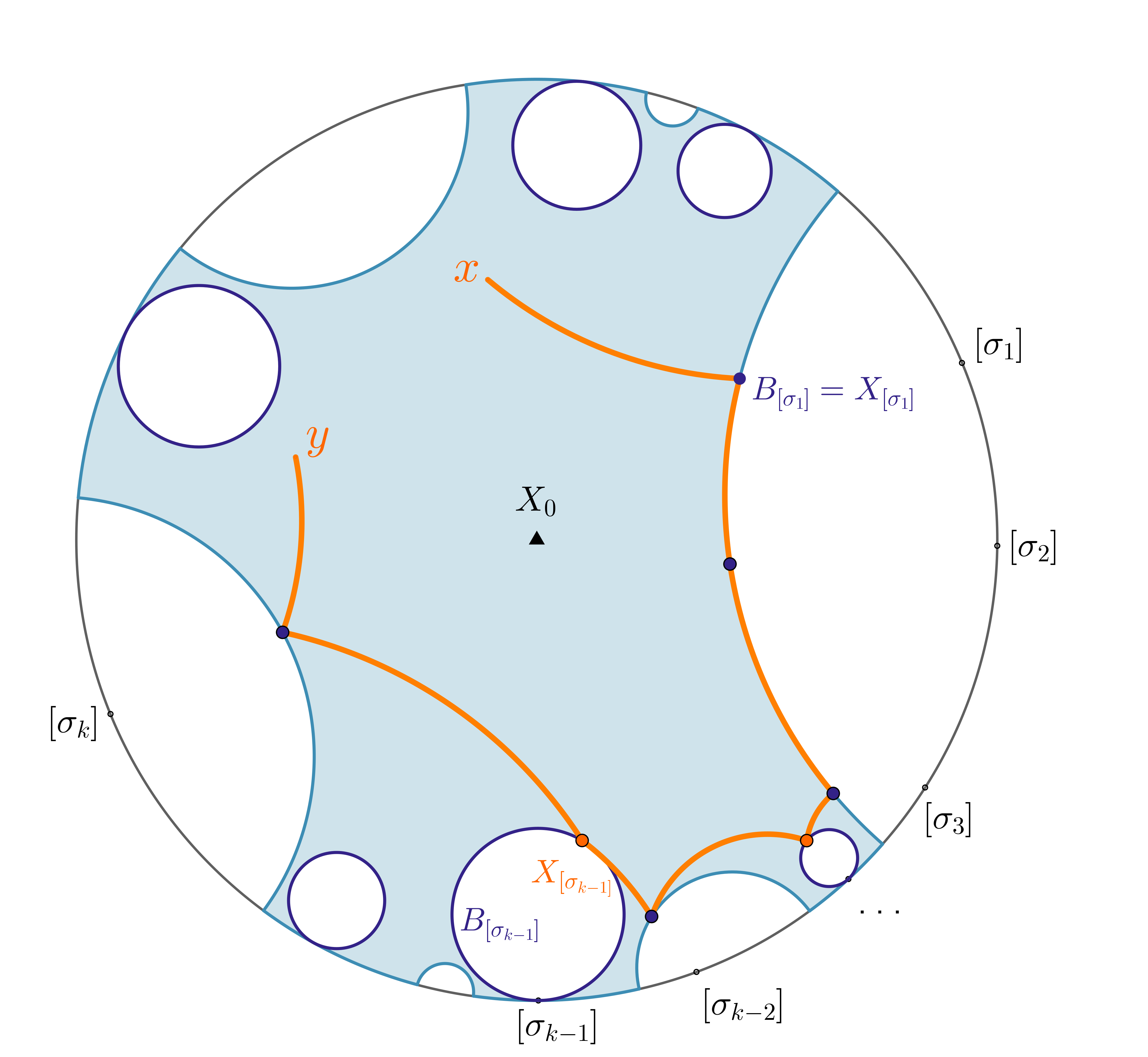}
    \caption{Examples of the constructions from \cref{def:horoballs} and \cref{def:preferred-paths}, projected to $D$. The region shaded in light blue, \emph{including} its boundary in $D$, is the truncated convex hull $\overline{D}$. 
    Some saddle connection directions are labeled along $\partial D$.
    For nonparabolic directions $\dir{\sigma_*}$---that is, those outside $\hull{G}$---the associated ``horopoint'' $B_{\dir{\sigma_*}}$ (dark blue dot) is the $\rho$-closest projection of $\dir{\sigma_*}$ onto $\partial \hull{G}$.
    For the remaining parabolic directions $\dir{\sigma_*}$, the associated ``horoball'' $B_{\dir{\sigma_*}}$ (dark blue circle and its interior) is a closed horoball in $\hull{G}$ based at $\dir{\sigma_*}$, with the choice of horoball made according to the technical requirements of \cref{def:horoballs}.
    Finally, for each $\sigma_*$ there is a choice of point $X_{\dir{\sigma_*}}$: When $\sigma_*$ is nonparabolic $X_{\dir{\sigma_*}} := B_{\dir{\sigma_*}}$ (dark blue dot), and when $\sigma_*$ is parabolic $X_{\dir{\sigma_*}} \in \partial B_{\dir{\sigma_*}}$ (orange dot).
    \\ \phantom{hello}
    To construct the preferred path from $x$ to $y$, $\varsigma(x,y)$, as in \cref{def:preferred-paths}, first observe that the geodesic between $f(x)$ and $f(y)$ in the fiber over $X_0$ (black triangle) consists of saddle connections $\sigma_1$ through $\sigma_k$, in order. The fact that their associated directions appear in a nice order along $\partial D$ is not a coincidence; see \cref{lem:structure-of-ideal-fans}. The preferred path is constructed from alternating concatenations of hyperbolic geodesics through the horizontal fibers of $\overline{E}$ (orange curves) with saddle connections traversed in the fibers over the chosen points $X_*$ (dark blue or orange dots). 
    Because this diagram shows only the projection of $\varsigma(x,y)$ to $D$, it is important to keep in mind that the horizontal pieces (orange curves) all belong to different horizontal fibers in $\overline{E}$ and that the saddle connections (dark blue or orange dots) in fact have positive length.
    }
    \label{fig:horostuff-paths}
\end{figure}

\subsubsection{Convex hull and truncated convex hull \texorpdfstring{$(\overline{D},\overline{\rho})$}{}}

Let $\hull{G}$ denote the convex hull of the limit set of $G$ as defined in \cref{sec:convex-hull-and-convex-core}.
    \nomenclature[hull(G)]{$\hull{G}$}{convex hull of the limit set of $G$}
Because $\hull{G}$ is a subset of $D$ and $(D,\rho)$ is a length space, the space $(\hull{G}, \rho\vert_{\hull{G}})$ is also a length space.

The \termttt{truncated convex hull} of $G$, denoted $\overline{D}$, is obtained from $\hull{G}$ by deleting the interiors of the horoballs $B_{\alpha}$. That is,
$$ \overline{D} := \hull{G} \backslash \bigcup_{\alpha \in \directions} B^\circ_{\alpha}, $$
where $B^\circ_{\alpha}$ denotes the interior of $B_{\alpha}$.
    \nomenclature[D bar]{$\overline{D}$}{truncated convex hull, obtained from $\hull{G}$ by deleting the interiors of the horoballs $B_\alpha$}
(Note that for nonparabolic $\alpha \in \directions$, $B_{\alpha}$ is a point and $B^\circ_{\alpha} = \emptyset$.)
By \cite[Theorem II.11.27]{bh-nonpositive-curvature}, the truncated disk $D \backslash \cup_{\alpha \in \directions} B^\circ_{\alpha}$ with the induced path metric is $\CAT{0}$.
Therefore the restriction to $\overline{D}$ with the associated restricted metric $\overline{\rho}$ is also a $\CAT{0}$ space.
    \nomenclature[rho bar]{$\overline{\rho}$}{induced path metric on the truncated convex hull $\overline{D}$}
Because $G$ is finitely generated, it acts cocompactly on $\overline{D}$ (refer to \cref{sec:convex-hull-and-convex-core}).

Denote by $\DtoDbar: D \rightarrow \overline{D}$ the $\rho$-closest-point projection of $D$ onto $\overline{D}$.
    \nomenclature[proj]{$\DtoDbar$}{$\rho$-closest-point projection of $D$ onto $\overline{D}$}

\subsubsection{Collapsed convex hull \texorpdfstring{$(\hat{D},\hat{\rho})$}{}}

Let $\hat{D}$ be the \termttt{collapsed convex hull}, the quotient space obtained from $\overline{D}$ by collapsing the boundary of each horoball $B_{\alpha}$ to a point.
    \nomenclature[D hat]{$\hat{D}$}{collapsed convex hull, quotient space obtained from $\overline{D}$ by collapsing each horoball $B_{\alpha}$ to a point}
The quotient pseudometric $\hat{\rho}$ is positive-definite, therefore a metric, and so the space $(\hat{D}, \hat{\rho})$ is a length space (\cref{lem:bh-length-space}).
    \nomenclature[rho hat]{$\hat{\rho}$}{quotient pseudometric on $\hat{D}$}

Analogous to the map $\DbartoDhat : \overline{D} \rightarrow \hat{D}$ which collapses each horoball to a point, there is also a map $\hulltoDhat: \hull{G} \rightarrow \hat{D}$ which collapses the \emph{interior} of each horoball to a point.
    \nomenclature[p]{$\DbartoDhat$}{map $\overline{D} \rightarrow \hat{D}$ which collapses each $B_\alpha$ to a point}
    \nomenclature[p]{$\hulltoDhat$}{map $\hull{G} \rightarrow \hat{D}$ which collapses the interior of each $B_\alpha$ to a point}

\subsection{Bundles}

\subsubsection{Total space \texorpdfstring{$E$}{E}}
\label{sec:total-space-E}

Let $\pi: E \rightarrow \hull{G}$ be the pullback bundle of the Bers fibration (\cref{eqn:bers-fibration}) via the inclusion $\hull{G} \subset D \subset \calT(S)$, identifying $E \subset \calT(\Dot{S})$. 
    \nomenclature[pi]{$\pi:E \rightarrow \hull{G}$}{pullback bundle of the Bers fibration via the inclusion $\hull{G} \subset \calT(S)$}
    \nomenclature[E]{$E$}{surface bundle over $\hull{G}$ given by $\pi$; also admits a product structure $E \cong \hull{G} \times \Tilde{S}$}
Let $\Gamma < \MCG{\Dot{S}}$ be the $\pi_1 S$-extension of $G$ (that is, the group fitting into the Birman exact sequence (\cref{eq:birman-exact-sequence})), which acts on $E$.
    \nomenclature[Gamma]{$\Gamma$}{extension group of $G$}

The space $E$ is a surface bundle over $\hull{G}$: By construction the fiber over $X \in \hull{G}$, denoted $E_X := \pi^{-1}(X) $, 
    \nomenclature[EX]{$E_X$}{fiber of $E$ over $X$, i.e. $\pi^{-1}(X)$}
is canonically identified with $\Tilde{S}$ equipped with the pulled back complex structure $X$ and flat metric $q_X$.
    \nomenclature[qX]{$q_X$}{flat metric associated to $X$, i.e. terminal flat metric from the Teichm\"uller mapping $(S,X_0,q_0) \rightarrow (S,X,q_X)$}
Recall that there was a fixed choice of a complex structure $X_0$ (which is necessarily in $\hull{G}$), and denote the fiber over $X_0$ by $E_0$.
    \nomenclature[E0]{$E_0$}{fiber of $E$ over $X_0$}

It will be necessary to have maps between fibers. For $X,Y \in \hull{G}$,
let $f_{X,Y} : E_Y \rightarrow E_X $ be the lift of the Teichm\"uller map, i.e., the map that sends $y \in E_Y$ to the unique point $f(y) \in E_X$ along the lift of the geodesic in $D$ connecting $\pi(x)$ and $\pi(y)$.
    \nomenclature[fXY]{$f_{X,Y}:E_Y \rightarrow E_X$}{lift of the Teichm\"uller map to the universal cover}
This map is affine with respect to the flat metrics $q_X$ and $q_Y$ and is $e^{\rho(X,Y)}$-bilipschitz by construction.
For any $X,Y,Z \in \hull{G}$, the composition $f_{X,Y} \circ f_{Y,Z}$ agrees with $f_{X,Z}$.
For any $X \in \hull{G}$, define $f_X : E \rightarrow E_X $ by $f_X \vert_{E_Y} = f_{X,Y}$ for $Y \in \hull{G}$.
    \nomenclature[fX]{$f_X:E \rightarrow E_X$}{map determined by varying $f_{X,Y}$ over all $Y \in D$}
When $X = X_0$, the map is denoted $f = f_{X_0} : E \rightarrow E_0$.
    \nomenclature[f]{$f:E \rightarrow E_0$}{alternative notation for $f_{X_0}$}

In fact, $E$ is a product: It admits a product structure $E \cong \hull{G} \times \Tilde{S}$. The map $\pi$ is the projection onto the first factor $\hull{G}$. The projection onto $\Tilde{S}$, the universal cover of $S$, is any map $f_X$ for $X \in \hull{G}$.
    \nomenclature[S tilde]{$\tilde{S}$}{universal cover of $S$}
Define a metric $d$ on $E$ as the orthogonal direct sum of the Poincar\'e metric $\rho$ in each horizontal fiber $D_x$ and the pulled back flat metric $q_X$ in each fiber $E_X \cong \Tilde{S}$.
    \nomenclature[d]{$d$}{metric on $E$}

For any $X \in \hull{G}$, denote by $\Sigma_X \subset E_X$ the set of cone points of the flat structure $q_X$ on $E_X$, and define
$$ \Sigma = \bigcup_{X \in \hull{G}} \Sigma_X .$$
    \nomenclature[SigmaX]{$\Sigma_X$}{for any $X \in \hull{G}$, the set of cone points of the flat structure $q_X$ on $E_X$}
    \nomenclature[Sigma]{$\Sigma$}{the set of all cone points in $E$, i.e. $\cup_{X \in \hull{G}} \Sigma_X$}

The quotient $E/\Gamma$ is generally a noncompact $S$-bundle over the convex core $\hull{G}/G$.
If the quotient $\hull{G}/G$ is compact and hence $E/\Gamma$ is compact, then this represents a special case in which there are no parabolic directions for $G$---that is, the Veech group $G$ is convex cocompact and therefore the extension group $\Gamma$ is hyperbolic (see \cites{fm-convex-cocompactness}{hamenstaedt-hyperbolic-extensions} or \cref{thm:new-old}).
In the general case, a $\Gamma$-equivariant quotient of $E$ is constructed from the $G$-equivariant quotient $\hull{G} \rightarrow \hat{D}$ as described in the following several subsections.

\subsubsection{Horizontal fibers \texorpdfstring{$D_x$}{}}
For any $X \in \hull{G}$ and $x \in E_X$, the horizontal fiber $D_x = f_X^{-1}(x)$ is the unique lift of $\hull{G}$ to the Teichm\"uller disk in $\calT(\dot{S})$ through $x$ that covers $\hull{G}$ via the projection $\pi$.
    \nomenclature[Dx]{$D_x$}{preimage $f_X^{-1}(x)$, i.e. the unique lift of $\hull{G}$ to the Teichm\"uller disk in $\calT(\Dot{S})$ through $x$ that covers $\hull{G}$ via the projection $\pi$}
\begin{remark}
A more appropriate choice of notation might be $\hull{G}_x$, to emphasize that this is a lift of only the convex hull rather than the entire Teichm\"uller disk $D$.
Because there will not be any objects lifted to the complement of the hull, the notation $D_x$ is used for simplicity.
\end{remark}

\subsubsection{Pullback bundle over truncated convex hull \texorpdfstring{$\overline{E}$}{}}
Define the pullback bundle over the truncated convex hull by
$$ \overline{E} = \pi^{-1} ( \overline{D} ) = E \backslash \bigcup_{\alpha \in \directions} \calB_{\alpha}^\circ,$$
where $\calB_{\alpha}^\circ = \pi^{-1}(B_{\alpha}^\circ)$ is the interior of the horoball preimage $\calB_{\alpha} = \pi^{-1}(B_{\alpha})$. 
    \nomenclature[B alpha circ]{$B^\circ_{\alpha}$}{interior of $B_{\alpha}$}
    \nomenclature[B alpha cal]{$\calB_{\alpha}$}{preimage $\pi^{-1}\left(B_{\alpha}\right)$}
    \nomenclature[B alpha cal circ]{$\calB^\circ_{\alpha}$}{interior of the horoball preimage $\calB_{\alpha}$}
(As before, if $\alpha$ is nonparabolic then $B_{\alpha}$ is a point and $B_{\alpha}^\circ = \emptyset$, so $\calB_{\alpha}^\circ = \emptyset$ also.)

The truncated bundle $\overline{E}$ is equipped with the length metric $\overline{d}$ induced from $(E,d)$. 
    \nomenclature[d bar]{$\overline{d}$}{metric on $\overline{E}$}
Because $G$ acts isometrically and cocompactly on $(\overline{D}, \overline{\rho})$ and $\pi_1(S)$ acts isometrically and cocompactly on $(\Tilde{S},\tilde{q_0})$, $\Gamma$ acts isometrically and cocompactly on $(\overline{E},\overline{d})$.
Therefore the space $(\overline{E}, \overline{d})$ is quasi-isometric to the group $\Gamma$ equipped with the word metric.
Then the quotient $\overline{E}/\Gamma$ is a compact $S$-bundle over $\overline{D}/G$,
$$ S \rightarrow \overline{E}/\Gamma \rightarrow \overline{D}/G. $$
As in \cite{ddls-extensions}, compactness of $\overline{E}/\Gamma$ gives that $\Sigma$ is $r$-dense in $\overline{E}$ for some $r > 0$. 
By an application of the Arzel\`a-Ascoli Theorem, $(\overline{E}, \overline{d})$ is a geodesic space.

\subsection{Electrified space \texorpdfstring{$\hat{E}$}{}}

\subsubsection{Bass-Serre trees}
\label{sec:bass-serre-trees}

Let $\alpha$ be a parabolic direction. 
The $\R$-tree dual to the foliation of $E_{X_{\alpha}}$ in the direction ${\alpha}$, denoted $T_\alpha$, is a weighted Bass-Serre tree when equipped with the metric defined by the transverse measure on the foliation of $E_{X_{\alpha}}$ in direction ${\alpha}$.
    \nomenclature[T alpha]{$T_\alpha$}{weighted Bass-Serre tree for a parabolic direction $\alpha$, i.e. the $\R$-tree dual to the foliation of $E_{X_{\alpha}}$ in the direction ${\alpha}$ }
Define a map 
$$t_{\alpha} : E \rightarrow T_{\alpha}$$ 
to be the composition of $f_{X_{\alpha}} : E \rightarrow E_{X_{\alpha}}$ followed by the projection $E_{X_{\alpha}} \rightarrow T_{\alpha} $.
    \nomenclature[t alpha]{$t_\alpha: E \rightarrow T_{\alpha}$}{composition of $f_{X_\alpha}$ followed by projection to $T_\alpha$}
Recall from \cref{sec:horoballs-and-horopoints} that there was a choice of fixed point $X_{\alpha} \in \partial B_{\alpha}$,
and note that the tree $T_{\alpha}$ is independent of the choice of $X_{\alpha} \in \partial B_{\alpha}$.
Therefore the map $t_{\alpha}$ is also independent of this choice.
Because there are only finitely many $\Gamma$-orbits of edges of these trees, every edge of every tree has length uniformly bounded above and below.

\subsubsection{Collapsed ``bundle'' \texorpdfstring{$\hat{E}$}{}} 
\label{sec:hatted-bundle}
Define $\hat{E}$ to be the quotient of $E$ obtained by collapsing each $\calB_{\alpha}$ onto $T_{\alpha}$ via the restriction $t_{\alpha} \vert_{\calB_{\alpha}}$, and denote this map by $P : E \rightarrow \hat{E}$.
    \nomenclature[P]{$P:E \rightarrow \hat{E}$}{quotient of $E$ obtained by collapsing each $\calB_{\alpha}$ onto $T_{\alpha}$ via the restriction $t_{\alpha} \vert_{\calB_{\alpha}}$}
    \nomenclature[E hat]{$\hat{E}$}{space obtained from $\overline{E}$ by collapsing horoballs to Bass-Serre trees via the map $P$}
Note that the quotient is $\Gamma$-equivariant because it is constructed from the $G$-equivariant quotient $\hull{G} \rightarrow \hat{D}$.
Finally, $\hat{E}$ is equipped with the quotient pseudometric $\hat{d}$ obtained from $\overline{d}$ by the restriction $ P \vert_{\overline{E}} : \overline{E} \rightarrow \hat{E}$.
    \nomenclature[d hat]{$\hat{d}$}{metric on $\hat{E}$}
As in \cite[Lemma 3.2]{ddls-extensions}, $\hat{d}$ is positive-definite, therefore a metric. 
Then by \cref{lem:bh-length-space}, the space $(\hat{E}, \hat{d})$ is a length space.  
The following fact is used throughout the paper.
\begin{lemma}[{\cite[Lemma 3.3]{ddls-extensions}}]
\label{lem:P-is-lipschitz}
The map $P$ is $1$-Lipschitz.
\end{lemma}

\section{Guessing Geodesics}
\label{sec:hyperbolicity}

Hyperbolicity of $\hat{E}$ will be proven using the ``guessing geodesics'' criterion of Bowditch \cite[Proposition 3.1]{bowditch:uniform-hyperbolicity} and Masur-Schleimer \cite[Theorem 3.15]{ms:geometry-of-disk-complex}, as formulated by \cite{ddls-extensions}. 

\begin{proposition}[\toolttt{Guessing geodesics}; {\cite[Proposition 2.2]{ddls-extensions}}]
\label{lem:guessing-geodesics}
Suppose $\Omega$ is a length space, $\Upsilon \subset \Omega$ an $R$-dense subset for some $R > 0$, and $\delta \geq 0$ a constant such that for all pairs $x,y \in \Upsilon$ there are rectifiably path-connected sets $L(x,y) \subset \Omega$ containing $x,y$ satisfying the properties:
\begin{enumerate}[(1)]
    \item the $L(x,y)$ form $\delta$-slim triangles, and
    \item if $x,y \in \Upsilon$ have $\hat{d}(x,y) \leq 3 R$, then the diameter of $L(x,y)$ is at most $\delta$.
\end{enumerate}
Then $\Omega$ is hyperbolic.
\end{proposition}

Saying the $L(x,y)$ form $\delta$-slim triangles means that for all $x,y,z \in \Upsilon$,
$$ L(x,y) \subset N_\delta \left( L(x,z) \cup L(z,y) \right), $$
where $N_\delta$ denotes the $\delta$-neighborhood.
    \nomenclature[N delta]{$N_\delta()$}{$\delta$-neighborhood}
This is similar to the usual slim triangles condition for hyperbolicity from \cref{sec:gromov-hyperbolicity}, except that the $L(x,y)$ need not be geodesics.

This section will verify most of the hypotheses required to apply the guessing geodesics criterion.
Many of the arguments made by \cite{ddls-extensions} continue to hold in this case.
However, the proof that the sets $L(x,y)$ (constructed in \cref{sec:preferred-paths}) form slim triangles differs considerably and composes the entirety of \cref{sec:slimness}.

\subsection{Preferred paths and collapsed preferred paths}
\label{sec:preferred-paths}

The sets $L(u,v)$ in the guessing geodesics criterion will be constructed from the following preferred paths in $E$ which connect cone points by concatenating alternating geodesics in the horizontal and vertical fibers.
\begin{definition}
\label{def:preferred-paths}
For two cone points $x,y \in \Sigma$ the \termttt{preferred path} from $x$ to $y$, denoted $\varsigma(x,y)$, is constructed as follows. 
    \nomenclature[sigma(x,y)]{$\varsigma(x,y)$}{preferred path between $x,y \in \Sigma$}

\begin{enumerate}
    \item In $E_0$, $f(x)$ and $f(y)$ are cone points with respect to the flat metric $q_0$, and they are connected by a geodesic segment which is a concatenation of saddle connections in $E_0$. That is,
    $$ [f(x),f(y)] = \sigma_1 \sigma_2 \cdots \sigma_k. $$
    Each saddle connection $\sigma_i$ is considered to be oriented with initial point $\sigma_i^- := \sigma_i \cap \sigma_{i-1}$ for $i=2,\dots,k$ (and $\sigma_1^- = f(x)$) and terminal point $\sigma_i^+ := \sigma_i \cap \sigma_{i+1}$ for $i=1, \dots, k-1$ (and $\sigma_k^+ = f(y)$). In the preferred path, each saddle connection $\sigma_i$ will be traversed in the fiber over the associated point $X_{\dir{\sigma_i}} \in B_{\dir{\sigma_i}}$ as fixed in \cref{sec:horoballs-and-horopoints}. 
    More precisely, define the $i$th \termttt{saddle piece} to be
    $$ \gamma_i := f_{X_{\dir{\sigma_i}}}(\sigma_i), $$
    which is a segment of the preferred path for all $1 \leq i \leq k$. 
        \nomenclature[gamma i]{$\gamma_i$}{$i$th saddle piece of $\varsigma(x,y)$}
    Each $\gamma_i$ is considered to be oriented with initial point
    $$ \gamma_i^- 
        := f_{X_{\dir{\sigma_i}}}\left( \sigma_i^- \right) $$
    and terminal point 
    $$ \gamma_i^+ 
        := f_{X_{\dir{\sigma_i}}}\left( \sigma_i^+ \right).  $$
    \item The \termttt{horizontal pieces} $h_i$ are chosen to make the preferred path continuous. 
        \nomenclature[h i]{$h_i$}{$i$th horizontal piece of $\varsigma(x,y)$}
    Specifically, for $1 \leq i \leq k-1$, $h_i$ is the geodesic in $D_{\gamma_i^+} = D_{\gamma_{i+1}^-}$ which connects $\gamma_i^+$ to $\gamma_{i+1}^-$. The first horizontal piece, denoted $h_0$, is the geodesic in $D_x = D_{\gamma_1^-}$ connecting $x$ to $\gamma_1^-$. The last horizontal piece, denoted $h_k$, is the geodesic in $D_{\gamma_k^+} = D_y$ connecting $\gamma_k^+$ to $y$.
    \item Define 
    $$ \varsigma(x,y) := h_0 \gamma_1 h_1 \gamma_2 h_2 \cdots \gamma_k h_k. $$
\end{enumerate}
See \cref{fig:horostuff-paths} for an example.
The \termttt{collapsed preferred paths}, denoted $\hat{\varsigma}(x,y)$, are the images of the preferred paths $\varsigma(x,y)$ under the map $P: E \rightarrow \hat{E}$; see \cref{sec:hatted-bundle}.
    \nomenclature[sigma(x,y) hat]{$\hat{\varsigma}(x,y)$}{collapsed preferred path between $x,y \in \Sigma$} 
\end{definition}

\subsection{The sets \texorpdfstring{$L(u,v)$}{L(u,v)}}
\label{sec:sets-l-u-v}

Let $\calV$ be the set of all vertices of all Bass-Serre trees in $\hat{E}$.
    \nomenclature[V]{$\calV$}{set of all vertices of all Bass-Serre trees $T_\alpha$}
Given $X \in \hull{G}$ and a saddle connection $\sigma$, the union of the saddle connections in $E_X$ with direction $\dir{\sigma}$ is precisely the preimage of the vertices of the Bass-Serre tree $T_{\dir{\sigma}}$ under the map
\begin{equation}
\label{eq:fibers-to-trees}
    E_X \xrightarrow{f_{X_{\dir{\sigma}}, X}} E_{X_{\dir{\sigma}}} \xrightarrow{P} T_{\dir{\sigma}},
\end{equation}
which is simply the restriction of $t_{\dir{\sigma}}^{-1}$ (see \cref{sec:bass-serre-trees}) to the fiber $E_X$.
For any vertex $v \in \calV$ and $X \in \hull{G}$, the preimage of $v$ under this composition (\ref{eq:fibers-to-trees}) is called the \termttt{$v$-spine in $E_X$} and is denoted $\theta^v_X$.
    \nomenclature[theta v X]{$\theta^v_X$}{$v$-spine in $E_X$}
Note that a choice of $v \in \calV$ determines the Bass-Serre tree $T_{\dir{\sigma}}$ to which $v$ belongs,
and denote the union of the $v$-spines over all $X \in \partial B_{\dir{\sigma}}$ by $\theta^v$.
    \nomenclature[theta v]{$\theta^v$}{union of $v$-spines over all fibers in the horoball associated to $v$}

Given $u,v \in \calV$, define
$$ L(u,v) = \bigcup \hat{\varsigma}(x,y), $$
where the union is taken over all $x \in \theta^u \cap \Sigma$ and $y \in \theta^v \cap \Sigma$.
    \nomenclature[L(u,v)]{$L(u,v)$}{union of $\hat{\varsigma}(x,y)$ over all $x \in \theta^u \cap \Sigma$ and $y \in \theta^v \cap \Sigma$}
Since $\varsigma(x,y)$ is a finite length path (by construction, each saddle piece and horizontal piece has finite length), $P$ is $1$-Lipschitz (\cref{lem:P-is-lipschitz}), and all $\hat{\varsigma}(x,y)$ in $L(u,v)$ connect $u$ to $v$, it follows that the set $L(u,v)$ is a path connected, rectifiable set containing both $u$ and $v$.

\subsection{Hyperbolicity}

The statement of the following theorem is identical to \cite[Theorem 4.2]{ddls-extensions}, but the proof depends on results which require considerable reformulation in the case that $G$ is only finitely generated. 
\begin{restatable*}[\theoremttt{Collapsed preferred paths form slim triangles}]{theorem}{collapsedpreferredpathsformslimtriangles}
    \label{thm:collapsed-preferred-paths-form-slim-triangles}
    There exists $\delta > 0$ so that collapsed preferred paths form $\delta$-slim triangles. That is, for any $x,y,z \in \Sigma$,
        $$ \hat{\varsigma}(x,y) \subset N_\delta \left( \hat{\varsigma}(x,z) \cup \hat{\varsigma}(y,z) \right). $$
\end{restatable*}
\noindent As the numbering suggests, the entirety of \cref{sec:slimness} is dedicated to the proof of this theorem. For now, \cref{thm:collapsed-preferred-paths-form-slim-triangles} is assumed and used to prove hyperbolicity of $\hat{E}$ in \cref{thm:hat-e-is-hyperbolic}.
Besides this, the last result needed is \cref{cor:g-g-pt-2}, which is stated after a few more constructions below.

\cref{def:combinatorial-path}, \cref{lem:combinatorial-criterion}, and \cref{lem:combinatorial-paths-bounded-length} do not reappear in this paper until \cref{sec:hhs}; since the proofs are purely technical rather than geometrically intuitive, the reader may want to skip these on a first pass.

\begin{definition}
\label{def:combinatorial-path}
A \termttt{horizontal jump} in $\hat{E}$ is the image under $\overline{P}$ of a geodesic in $\overline{D}_z$, for some $z \in \Sigma $, that connects two components of $\partial \overline{D}_z$ and whose interior is disjoint from $\partial \overline{D}_z$. 
A \termttt{combinatorial path} in $\hat{E}$ is a concatenation of horizontal jumps and nonparabolic saddle connections; in particular, it is a concatenation of the collapsed preferred paths from \cref{def:preferred-paths}.
\end{definition}
\noindent While this paper leaves the definition of a horizontal jump unchanged from {\cite[Definition 3.7]{ddls-extensions}} (although the construction of $\overline{D}_z$ differs; see \cref{sec:construction}), the definition of a combinatorial path has been changed to fit the next statement, which is adapted from {\cite[Lemma 3.8]{ddls-extensions}}; their argument is sketched here and supplemented as necessary. 
\begin{lemma}
\label{lem:combinatorial-criterion}
There is a constant $C > 0$ such that any pair of points $x,y \in \calV$ may be connected by a combinatorial path of length at most $C \hat{d}(x,y)$.
In particular, this combinatorial path consists of at most $2 C \hat{d}(x,y)$ horizontal jumps and saddle connections.
\end{lemma}

\begin{proof}[Proof sketch]
Denote ${\Sigma}_\alpha := \Sigma \cap \partial \calB_\alpha$.
There exists some $K \geq 30$ so that for each $\alpha \in \directions$ and $z \in \partial \calB_\alpha$ there is $w \in {\Sigma}_\alpha$ with $\overline{d}(z,w) \leq K/30$ (in particular, $K$ is chosen to be $30M$, where $M$ is from \cite[Lemma 3.4(3)]{ddls-extensions}).
The next step in the argument is to describe how to traverse between points of $\overline{P}({\Sigma}_\alpha)$.
Recall from \cref{sec:bass-serre-trees} that $T_\alpha$ is the $\R$-tree dual to the foliation of $E_{X_\alpha}$, and denote the length metric on $T_\alpha$ by $\ell_\alpha$.
    \nomenclature[l alpha]{$\ell_\alpha$}{length metric on $T_\alpha$}
The statement of the following claim is identical to {\cite[Claim 3.9]{ddls-extensions}}; this is the only part of the argument for which the proof does not extend verbatim to the non-lattice case.

\begin{claim}
\label{lem:combinatorial-paths-bounded-length}
There exists $R'>0$ such that any pair of points $v_1, v_2 \in \overline{P}({\Sigma}_\alpha) \subset \calV$ may be connected by a combinatorial path of length at most $R' \ell_\alpha(v_1,v_2)$.
\end{claim}
\begin{proof}
It suffices to assume that $v_1$ and $v_2$ are adjacent vertices of $T_\alpha$.
Choosing any $X \in \partial B_\alpha$, the preimages of $v_1$ and $v_2$ under the restriction $\overline{P}\vert_{E_X}: E_X \rightarrow T_\alpha$ are adjacent spines $\theta_X^{v_1}$ and $\theta_X^{v_2}$ which are separated by a strip of uniformly bounded width (\cite[Lemma 3.4]{ddls-extensions}).
Then there exists a saddle connection $\sigma \subset E_X$ of bounded length which joins cone points $y_1 \in \theta_X^{v_1}$ and $y_2 \in \theta_X^{v_2}$. Denote the direction of $\sigma$ by $\dir{\sigma} \in \directions$; unlike the case of \cite{ddls-extensions}, this direction may be nonparabolic.

Because $\sigma$ has bounded length, $X \in \partial B_\alpha$ lies within a bounded distance of $Y \in \partial B_{\dir{\sigma}}$.
Let $z_i := f_{Y,X}(y_i)$. 
Let $h_i$ be the horizontal geodesic in $D_{y_i}$ from $y_i$ to $z_i$, and note that the $\overline{P}$-image of each component of $h_i \cap \overline{D}_{y_i}$ is a horizontal jump in $\hat{E}$. 
In particular, each $h_i \cap \overline{D}_{y_i}$ must have total length at most $\rho(X,Y)$, which is uniformly bounded.
The saddle connection $f_{Y,X}(\sigma) \subset E_Y $ has length bounded by $e^{\rho(X,Y)}$ times the length of $\sigma$, which are both bounded.
Therefore following the horizontal jumps along $h_1$, traversing $f_{Y,X}(\sigma)$, and then following the horizontal jumps along $h_2$ gives a bounded length combinatorial path in $\hat{E}$ from $v_1$ to $v_2$.
The claim follows from the fact that an edge in $T_\alpha$ must be at least some minimal length, so $\ell_\alpha(v_1,v_2)$ is uniformly bounded below.
\end{proof}

Note that there exists a constant $R_0$ so that $\calV$ is $R_0$-dense in $\hat{E}$ because $\calV$ is $\Gamma$-invariant and $\hat{E}/\Gamma$ is compact (since $\Gamma$ acts cocompactly on $\overline{E}$ and $\hat{E} / \Gamma$ is the continuous image of $\overline{E}/ \Gamma$ under the descent of $P:\overline{E} \rightarrow \hat{E}$; \cite[Lemma 3.6]{ddls-extensions}).

Now it can be shown that if $x,y \in \calV$ and $r = \hat{d}(x,y) >0$, then $x$ and $y$ can be connected by a combinatorial path of length $\leq 4 R' e^{Kr} r$ by \cite[Claim 3.10]{ddls-extensions}.
Set $C := 36 R' e^{3KR_0}$.
Then, if $r := \hat{d}(x,y) \leq 3R_0$, there is a combinatorial path joining $x$ and $y$ of length at most
$$4 R' e^{3KR_0} \hat{d}(x,y) = \frac{C}{9} \hat{d}(x,y),$$ 
satisfying the lemma.
Otherwise if $r > 3R_0$, $x$ and $y$ can be joined by a path $\gamma$ of length at most $2r$, which can be subdivided into $n = \lceil \mathrm{length}(\gamma) / R_0 \rceil$ equal-length subsegments of length at most $R_0$.
Since $\calV$ is $R_0$-dense, there is a sequence $\set{x_i} \subset \calV$ with $x_0=x$, $x_n=y$, and $\hat{d}(x_i, x_{i+1})$. Each pair $x_i$, $x_{i+1}$ can then be connected by a combinatorial path of length at most $CR_0/3$, and so there is a combinatorial path from $x$ to $y$ of length at most $CR_0n/3$. 
Finally, because $R_0n < \mathrm{length}(\gamma) + R_0 \leq 3\hat{d}(x,y)$, the first statement of the lemma holds.

The number of horizontal jumps and saddle connections in the combinatorial path joining $x$ and $y$ can be bounded as follows. 
Each horizontal jump has length at least $1$ because the horoballs associated to parabolic directions were constructed to be $1$-separated from each other and from $\partial \hull{G}$. 
Therefore there can be at most $C \hat{d}(x,y)$ horizontal jumps in the combinatorial path. 
Since the combinatorial path is a concatenation of collapsed preferred paths, it can include no more saddle connections than there are horizontal jumps.
\end{proof}

Combined with \cref{thm:collapsed-preferred-paths-form-slim-triangles}, this bound on $\hat{d}(x,y)$ allows for a bound on the diameter of the sets $L(u,v)$. The following statement is identical to {\cite[Lemma 4.4]{ddls-extensions}}; the proof is nearly identical except for some changes in the language regarding combinatorial paths, which were given a different defintion in this paper (see \cref{def:combinatorial-path} above).

\begin{corollary}
\label{cor:g-g-pt-2}
There exists a constant $C > 0$ so that if $u, v \in \calV$ with $\hat{d}(u,v) \leq 3R_0$ (where $R_0$ is a constant so that $\calV$ is $R_0$-dense in $\hat{E}$, as in the proof of \cref{lem:combinatorial-paths-bounded-length}), then $\mathrm{diam}(L(u,v)) \leq C$.
\end{corollary}
\begin{proof}
Pick $x \in \theta^u \cap \Sigma$ and $y \in \theta^v \cap \Sigma$. 
By \cref{lem:combinatorial-criterion} there is a combinatorial path from $P(x) = u$ to $P(y) = v$ of bounded length which is a concatenation of $n$ horizontal jumps and saddle connections, where the bound on the length and the number $n$ depend only on $R_0$. 
By repeatedly applying \cref{thm:collapsed-preferred-paths-form-slim-triangles}, $\hat{\varsigma}(x,y)$ is in the $n\delta$-neighborhood of the combinatorial path joining $P(x)= u$ to $P(y)=v$.
Therefore $\hat{\varsigma}(x,y)$ has uniformly bounded diameter.
As in \cite[Lemma 4.3]{ddls-extensions}, $L(u,v)$ is contained in the $2\delta$-neighborhood of $\hat{\varsigma}(x,y)$, where $\delta > 0$ is the constant from \cref{thm:collapsed-preferred-paths-form-slim-triangles}. Therefore $L(u,v)$ has uniformly bounded diameter also.
\end{proof}

Finally, it is possible to establish hyperbolicity of $\hat{E}$.

\begin{theorem}
\label{thm:hat-e-is-hyperbolic}
The space $(\hat{E}, \hat{d})$ is hyperbolic when the limit set $\Lambda(G)$ contains at least one parabolic point.
\end{theorem}
\begin{proof}
Hyperbolicity of $\hat{E}$ will be proven by the guessing geodesics criterion (\cref{lem:guessing-geodesics}).
It was established during construction that $(\hat{E}, \hat{d})$ is a length space (see \cref{sec:hatted-bundle} and \cite[Lemma 3.2]{ddls-extensions}).
Let $\calV \subset \hat{E}$ be the collection of all vertices of Bass-Serre trees in $\hat{E}$.
Let $R_0 > 0$ be as in the proof of \cref{lem:combinatorial-paths-bounded-length}, so that $\calV$ is $R_0$-dense in $\hat{E}$.
For any $u,v \in \calV$, $L(u,v)$ is a rectifiably path-connected set containing $u$ and $v$ (see \cref{sec:sets-l-u-v}).
Let $\delta > 0$ be the constant from \cref{thm:collapsed-preferred-paths-form-slim-triangles}.
As in \cite[Lemma 4.3]{ddls-extensions}, each $L(u,v)$ is contained in the $2\delta$-neighborhood of $\hat{\varsigma}(x,y)$ for any $x \in \theta^u \cap \Sigma$ and $y \in \theta^v \cap \Sigma$; because triangles of collapsed preferred paths in $\hat{E}$ are $\delta$-slim (\cref{thm:collapsed-preferred-paths-form-slim-triangles}), the $L(u,v)$ also form $3\delta$-slim triangles, satisfying condition (1) of the guessing geodesics criterion.
Finally, by \cref{cor:g-g-pt-2} the sets $L(u,v)$ have diameter bounded by a constant $C>0$ whenever $\hat{d}(u,v) \leq 3 R_0$, satisfying condition (2) of the guessing geodesics criterion.
Therefore $(\hat{E}, \hat{d})$ is hyperbolic.
\end{proof}

The action of $\Gamma$ on the total space $E$ is isometric by construction but non-cocompact in general.
The truncated convex hull $\overline{E}$ was constructed as a subspace on which $\Gamma$ acts cocompactly in addition to isometrically.
By a direct application of the Schwarz-Milnor lemma, $\overline{E}$ and $\Gamma$ are quasi-isometric.
However, neither is hyperbolic unless there are no parabolic directions on $S$.
To construct a hyperbolic space from $E$, for parabolic saddle connections $\sigma$ the sets $\calB_{\dir{\sigma}}$ are collapsed to Bass-Serre trees.
In $\Gamma$, this corresponds to coning off by the \termttt{vertex subgroups}, the stabilizers of the vertices of the Bass-Serre trees under the isometric action of $\Gamma$.
Then when the limit set $\Lambda(G)$ contains at least one parabolic point, the main theorem follows.
The proof is essentially identical to that of {\cite[Corollary 4.5]{ddls-extensions}}, but it is included here for the sake of completeness.
\ddlsgeneralization
\begin{proof}
    Recall from \cref{sec:hatted-bundle} that $\hat{E}$ is a length space.
    By construction, the action of $\Gamma$ on $\overline{E}$ is isometric and cocompact, and therefore so is the action of $\Gamma$ on $\hat{E}$.
    Let $\Upsilon_1, \dots, \Upsilon_k < \Gamma$ be representatives of the conjugacy classes of vertex subgroups stabilizing the vertices of the Bass-Serre trees.
    Any point-stabilizers for the action of $\Gamma$ on $\hat{E}$ are trivial or conjugate into one of the $\Upsilon_*$, and therefore any point in $\hat{E}$ has a discrete orbit under $\Gamma$.
    By \cref{thm:groupy-schwarz-milnor} it follows that, for any finite generating set $\calS$ of $\Gamma$, the coned-off Cayley graph $\mathrm{Cay}(\Gamma, \calS \cup \bigcup \Upsilon_i)$ is quasi-isometric to $\hat{E}$.
    Finally, because $\hat{E}$ is hyperbolic by \cref{thm:hat-e-is-hyperbolic}, so too is the coned-off Cayley graph of $\Gamma$.
\end{proof}

\section{Slim Triangles of Collapsed Preferred Paths}
\label{sec:slimness}

\cref{sec:hyperbolicity} verified all but one condition for the guessing geodesics criterion (\cref{lem:guessing-geodesics}). This section will verify the remaining condition, which is that all triangles of collapsed preferred paths in $\hat{E}$ are slim (\cref{thm:collapsed-preferred-paths-form-slim-triangles}). 

The following definitions are comparable to those in {\cite[Section 4.3]{ddls-extensions}}.

\begin{definition}
Given $x,y,z \in \Sigma$, define the associated geodesic triangle in $E_0$ by
$$ \Delta(x,y,z) := \left[ f(x), f(y) \right] \cup \left[ f(y), f(z) \right] \cup \left[ f(z), f(x) \right], $$
\nomenclature[Delta(x,y,z)]{$\Delta(x,y,z)$}{triangle of geodesics in $E_0$ with vertices $f(x)$, $f(y)$, $f(z)$ for $x,y,z \in \Sigma$}
the triangle of preferred paths in $E$ by 
$$ \Delta^{\varsigma}(x,y,z) := [\varsigma(x,y)] \cup [\varsigma(y,z)] \cup [\varsigma(z,x)], $$
    \nomenclature[Delta varsigma (x,y,z)]{$\Delta^{\varsigma}(x,y,z)$}{triangle of preferred paths in $E$ with vertices $x,y,z \in \Sigma$}
and the triangle of collapsed preferred paths in $\hat{E}$ by 
$$ \Delta^{\hat{\varsigma}}(x,y,z) := [\hat{\varsigma}(x,y)] \cup [\hat{\varsigma}(y,z)] \cup [\hat{\varsigma}(z,x)]. $$
    \nomenclature[Delta varsigma hat(x,y,z)]{$\Delta^{\hat{\varsigma}}(x,y,z)$}{triangle of collapsed preferred paths in $\hat{E}$ with vertices $x, y, z \in \Sigma$}

\noindent Equivalently, $\Delta^{\hat{\varsigma}}(x,y,z)$ is the image of $\Delta^\varsigma(x,y,z)$ under the map $P$.
If any pair of the sides of $\Delta(x,y,z)$ intersect at one or more (non-trivial) saddle connections, then $\Delta(x,y,z)$ is called \termttt{degenerate}. Otherwise, if each pair of sides of $\Delta(x,y,z)$ intersect only at the vertices, then $\Delta(x,y,z)$ is called \termttt{nondegenerate}.
The triangles $\Delta^{\varsigma}(x,y,z)$ and $\Delta^{\hat{\varsigma}}(x,y,z)$ are described as degenerate (or nondegenerate) if the triangle $\Delta(x,y,z)$ is degenerate (or nondegenerate).
\end{definition}

Recall from the construction that $E_0 \cong \tilde{S}$ (equipped with the lift of the flat metric $q_0$). Since $E_0$ is complete, simply connected, and nonpositively curved, it is $\CAT{0}$ (as a result of the Cartan-Hadamard theorem) and therefore it is uniquely geodesic.
Then the nontrivial intersection of two sides of a degenerate triangle must be a concatenation of saddle connections based at the shared vertex. 
As a consequence, every degenerate triangle---whether in $E_0$, $E$, or $\hat{E}$---contains a subtriangle which is nondegenerate.
See \cref{fig:degenerate-triangle}.

\begin{figure}[h]
\begin{center}
\begin{tikzpicture}[scale = .5]
\draw[ultra thick] (0,0) --  (5,1) -- (10,-1) -- (7,1) -- (6,2) -- (5,6) -- (4.5,4) -- (3,1.5) -- (0,0); 
\draw (5,6) -- (5,6.5) -- (4.5,7) -- (5,8);
\draw (10,-1) -- (10.5,-1.25) -- (11,-.5) -- (11.5,-1) -- (12,-1);
\draw (0,0) -- (-.5,0) -- (-1,.5) -- (-1.5,0) -- (-2,0) -- (-2.5,-.5) -- (-3,-.5);
\end{tikzpicture}
\caption{Illustration from \cite[Figure 2]{ddls-extensions} of a degenerate triangle in $E_0$ with the nondegenerate subtriangle in bold.}
\label{fig:degenerate-triangle}
\end{center}
\end{figure}
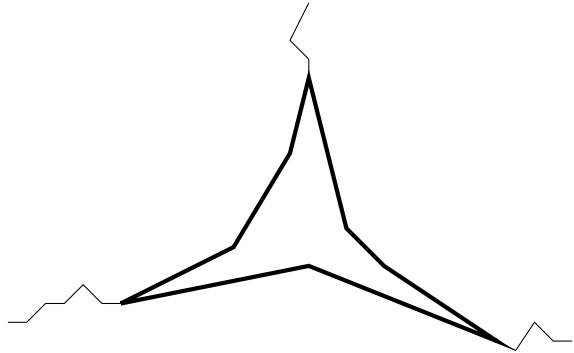

To prove that all triangles are slim, it suffices to prove that all nondegenerate triangles are slim (see \cite[Lemma 4.6]{ddls-extensions} for a detailed proof). 
Therefore the rest of this section assumes that all triangles are nondegenerate.

\subsection{Euclidean triangles}

\begin{definition}
Given $x,y,z \in \Sigma \cap E_0$, the triangle $\Delta(x,y,z)$ is called \termttt{Euclidean} if it is nondegenerate and each side consists of a single saddle connection.
When the information about $x,y,z$ is clear from context, a Euclidean triangle is sometimes denoted simply by $T$.
\end{definition}

\begin{proposition}[\toolttt{no short saddle connections condition}]
\label{cor:no-short-saddle-connections-in-truncated-hull}
For any $X \in \overline{D}$, there is a positive lower bound on the lengths of all saddle connections in $(S,X)$.
\end{proposition}
\begin{proof}
Recall from \cref{def:horoballs} that for parabolic saddle connections $\sigma$ the horoballs $B_{\dir{\sigma}}$ were chosen so that over $\partial B_{\dir{\sigma}}$ the saddle connection $\sigma$ would have (positive) length less than any saddle connection in a different direction. (The lengths of the nonparabolic saddle connections are necessarily positive over the convex hull.)
Because the length of the shortest saddle connection on $(S,X)$ is a continuous function on $D/G$ and $\overline{D}/ G$ is compact, there must be a positive lower bound on the lengths of all saddle connections in the fibers over $\overline{D}$.
\end{proof}

It is also true that Euclidean triangles cannot be too large. 
This is known for lattice surfaces (see for instance \cite{vorobets-planar-structures}), but requires a new proof for general flat surfaces.

\begin{proposition}[\toolttt{no large triangles condition}]
\label{lem:euclidean-triangle-bounded-area}
There exists $A > 0$ so that for any $X \in D$ the area of any Euclidean triangle in $(S,X,q)$ is at most $A$.
\end{proposition}
\begin{proof}
Any Euclidean triangle minus its vertices is in the flat surface $(S,X,q)$ is embedded (\cite[Lemma 2.1]{nguyen:topological-veech-dichotomy}).
Choose $A$ to be the area of $(S,X,q)$.
Then any Euclidean triangle in $(S,X,q)$ has area at most $A$.
\end{proof}

\begin{corollary}
Every Euclidean triangle is uniformly slim.
\end{corollary}
\begin{proof}
The inradius of any Euclidean triangle is bounded above by the inradius of an equilateral triangle of area $A$ (where $A$ is the constant from \cref{lem:euclidean-triangle-bounded-area}), so every Euclidean triangle is $2\sqrt{A}/ 3^{3/4}$-slim.
\end{proof}

Throughout this section, it will be necessary to relate Euclidean triangles in $E_0$ to particular ideal triangles in $D$, described below.
\begin{definition}
Let $x,y,z \in \Sigma$ so that $T = \Delta(x,y,z)$ is a Euclidean triangle consisting of saddle connections $\sigma_x$, $\sigma_y$, and $\sigma_z$. 
The \termttt{ideal triangle associated to $T$}, denoted $\tilde{T}$, is the ideal triangle in $D$ with vertices $\dir{\sigma_x}$, $\dir{\sigma_y}$, and $\dir{\sigma_z}$.  
\end{definition}

\noindent These are studied in depth by \cite{nguyen:topological-veech-dichotomy}.

\begin{definition}
Let $\Delta(x,y,z)$ be a Euclidean triangle. The \termttt{balance point of $\Delta(x,y,z)$} is the unique point $b \in D$ such that $f_b\left( \Delta(x,y,z) \right)$ is an equilateral triangle. 
\nomenclature[b]{$b$}{balance point of a triangle}
The \termttt{adjusted balance point of $\Delta(x,y,z)$}, denoted $\adj{b}$, is the projection of $b$ onto $\hull{G}$. 
\nomenclature[b]{$\adj{b}$}{adjusted balance point of a triangle---that is, the projection of the balance point $b$ to $\hull{G}$}
If $b \in \hull{G}$, then $b$ and $\adj{b}$ coincide.
\end{definition}

\begin{remark}
Existence and uniqueness of the balance point is given by \cite[Lemma 4.9]{ddls-extensions}.
Uniqueness of the adjusted balance point follows from uniqueness of the closest-point projection.  
\end{remark}

\begin{proposition}
\label{rmk:balance-pt-is-incenter}
Let $T = \Delta(x,y,z)$ be a Euclidean triangle with sides consisting of saddle connections $\sigma_x$, $\sigma_y$, and $\sigma_z$. 
Then the balance point $b$ is the incenter of the ideal triangle $\tilde{T}$  (that is, the center of the largest circle inscribed in $\tilde{T}$).
\end{proposition}

\begin{proof}
Without loss of generality assume that $\dir{\sigma_x}$ is the vertical direction (perhaps after acting on $D$ with an element of $\mathrm{SO}(2)$; see \cref{sec:teichmueller-space}).
Changing the coordinate chart by a rotation of $\pm\pi/3$ makes $\dir{\sigma_y}$ or $\dir{\sigma_z}$ vertical.
Therefore the three geodesic rays in $D$ starting at $b$ and ending at $\dir{\sigma_x}$, $\dir{\sigma_y}$, or $\dir{\sigma_z}$ diverge from $b$ at equal angles.
In particular, changing the coordinate chart by a rotation of $\pm\pi/3$ preserves $\tilde{T} \subset D$.
It also preserves the three geodesic segments from $b$ to the nearest point on each side of $\tilde{T}$,
so $b$ must be the incenter of $\tilde{T}$.
\end{proof}

\begin{corollary}
\label{cor:length-of-saddle-connection-bounded-at-balance-point}
Let $T$ be a Euclidean triangle, and let $b$ be the balance point of $T$. Because the area of $f_b(T)$ is bounded above by $A$ (\cref{lem:euclidean-triangle-bounded-area}), each side of $f_b(T)$ has length at most $2\sqrt{A}$.
\end{corollary}

For many of the following proofs it is important to recall from \cref{sec:basics} that given a saddle connection $\sigma$ with direction $\dir{\sigma} \in \partial D$, the horoballs based at $\dir{\sigma} \in D$ are sublevel sets for the length of $\sigma$, where the length of $\sigma$ shrinks exponentially along any geodesic ending at $\dir{\sigma}$.
The following technical lemma states this more precisely.

Throughout \cref{lem:bounded-ideal-triangle-edges} and \cref{lem:ideal-triangles-project-to-bounded-diameter-subsets}, a geodesic between $\dir{\sigma_x}$ and $\dir{\sigma_y}$ is denoted $\geod{\dir{\sigma_x}}{\dir{\sigma_y}}$.
    \nomenclature[(sigma,sigma)]{$\geod{\dir{\sigma_x}}{\dir{\sigma_y}}$}{hyperbolic geodesic segment between $\dir{\sigma_x}$ and $\dir{\sigma_y}$}

\begin{lemma}
\label{lem:bounded-ideal-triangle-edges}
Let $T = \Delta(x,y,z)$ be a Euclidean triangle with sides consisting of saddle connections $\sigma_x$, $\sigma_y$, and $\sigma_z$. Then for any segment $c \subset \geod{\dir{\sigma_x}}{\dir{\sigma_y}}$ of length $L > 0$,
$$ \min_{t \in c} \left(\min\left\{\length{f_t(\sigma_x)},\length{f_t(\sigma_y)}\right\} \right) \leq 2 \sqrt{3 A}e^{-L/2}, $$
where $A$ is the constant from \cref{lem:euclidean-triangle-bounded-area}.
\end{lemma}

\begin{proof}
At the balance point $b$ for $T$, the lengths of $f_b(\sigma_x)$ and $f_b(\sigma_y)$ are at most $2\sqrt{A}$ (\cref{cor:length-of-saddle-connection-bounded-at-balance-point}).
Denote by $t'$ the point of intersection of $\geod{\dir{\sigma_x}}{\dir{\sigma_y}}$ with the inscribed circle of $\Tilde{T}$.
The distance between $b$ and $t'$ is $\ln{\sqrt{3}}$ (as this is the inradius of an ideal triangle---see \cref{rmk:balance-pt-is-incenter}), so the lengths of $f_{t'}(\sigma_x)$ and $f_{t'}(\sigma_y)$ are at most $2\sqrt{A} e^{\ln{\sqrt{3}}} = 2 \sqrt{3A}$. 
Traversing $\geod{\dir{\sigma_x}}{\dir{\sigma_y}}$ towards $\dir{\sigma_x}$ (or towards $\dir{\sigma_y}$) shrinks $\length{\sigma_x}$ (or $\length{\sigma_y}$) exponentially as a function of the distance traversed.
Since one of the endpoints has distance at least $L/2$ from $t'$, it follows that $\min\left(\length{\sigma_x},\length{\sigma_y}\right) \leq 2\sqrt{3A} e^{-L/2}$ on $c$.
\end{proof}

\begin{lemma}
\label{lem:ideal-triangles-project-to-bounded-diameter-subsets}
Let $T = \Delta(x,y,z)$ be a Euclidean triangle consisting of saddle connections $\sigma_x$, $\sigma_y$, and $\sigma_z$; let $\tilde{T} \subset D $ be the ideal triangle associated to directions $\dir{\sigma_x}$, $\dir{\sigma_y}$, and $\dir{\sigma_z}$; and let $\DtoDbar(\tilde{T})$ denote the closest-point projection of $\tilde{T}$ onto $\overline{D}$ (where this map is the same as in \cref{sec:construction}).
Then the perimeter of $\DtoDbar(\tilde T)$ is uniformly bounded.
As a consequence, the diameter of $\DtoDbar(\tilde{T})$ is uniformly bounded.
\end{lemma}

\begin{proof}
Denote the balance point for $\tilde{T}$ by $b$ and recall that the lengths of $\sigma_x$, $\sigma_y$, and $\sigma_z$ are all uniformly bounded at $b$.
They are also uniformly bounded at the point closest to $b$ on the side of $\tilde{T}$ opposite (WLOG) $\dir{\sigma_x}$ (as in the proof of \cref{lem:bounded-ideal-triangle-edges}).
Then, traversing the geodesic from this point toward (WLOG) $\dir{\sigma_y}$, $\ell(\sigma_y) \to 0$.
Because there are no short saddle connections in $\overline{D}$ (\cref{cor:no-short-saddle-connections-in-truncated-hull}), there is a bound on the length of the intersection of each side of $\tilde{T}$ with $\overline{D}$.

If $\dir{\sigma}$ is a parabolic vertex of $\tilde{T}$, then $\DtoDbar(\tilde{T} \cap B_{\dir{\sigma}}) = \tilde{T} \cap \partial B_{\dir{\sigma}}$, which has uniformly bounded length due to the choice of horoballs (see \cref{def:horoballs}).

If the distance between $\tilde{T}$ and $\hull{G}$ is at least $1$, then the projection of $\tilde{T}$ onto any component of $\partial\hull{G}$ is bounded.
More generally, any component of $\tilde{T}$ which lies outside of the $1$-neighborhood of $\hull{G}$ in $D$ projects to a bounded subsegment of $\partial \hull{G}$. 

It only remains to consider the portions of $\tilde{T}$ contained in the $1$-neighborhood of $\hull{G}$.
By convexity of $\hull{G}$, the intersection of an side of $\tilde{T}$ with the $1$-neighborhood of $\hull{G}$ consists of at most one component.
Fix some $L>0$ and assume (without loss of generality) that there is such a component $c \subset \geod{\dir{\sigma_x}}{\dir{\sigma_y}}$ of length $L>0$ contained in the $1$-neighborhood of a component of $\partial \hull{G}$.
Then the lengths of $\sigma_x$ and $\sigma_y$ in the fibers over points of $\DtoDbar(c)$ are bounded by $2\sqrt{3A} e^{-L/2 + 1}$ (\cref{lem:bounded-ideal-triangle-edges}). Again, because there are no short saddle connections in $\overline{D}$ (\cref{cor:no-short-saddle-connections-in-truncated-hull}), this implies that the projection of $c$ to $\partial \hull{G}$ has bounded length.

Because each of these cases can occur at most three times for any $\tilde{T}$, there is a bound on the perimeter of $\DtoDbar(\tilde T) \cap \overline{D}$. Therefore the diameter of $\DtoDbar(\tilde T) \cap \overline{D}$ is also bounded.
\end{proof}

The next lemma shows that for any ideal triangle that intersects $\hull{G}$, the balance point and adjusted balance point are close to each other, to the horopoints associated to any nonparabolic vertices, and to the horoballs associated to any parabolic vertices. This statement functions similarly to \cite[Lemma 4.10]{ddls-extensions}, which is a consequence of \cite{vorobets-planar-structures}; both assume in addition that $G$ is a lattice, which is not necessary for the result below.

\begin{lemma}
\label{lem:about-balance-points}
Let $T = \Delta(x,y,z)$ be a Euclidean triangle consisting of saddle connections $\sigma_x$, $\sigma_y$, and $\sigma_z$ so that $\tilde{T} \cap \hull{G} \neq \emptyset$. 
Then the balance point $b$ of $\tilde{T}$, the adjusted balance point $\adj{b}$ of $\tilde{T}$, the horoballs associated to any parabolic vertices of $\tilde{T}$, and the horopoints associated to any nonparabolic vertices of $\tilde{T}$ are all uniformly close to one another.
\end{lemma}

\begin{proof}
Because the projection of $\tilde{T}$ to $\hull{G}$ has bounded diameter (\cref{lem:ideal-triangles-project-to-bounded-diameter-subsets}), the adjusted balance point $\adj{b}$ and the horopoints associated to any nonparabolic vertices of $\tilde{T}$ must be uniformly close.
Also, the adjusted balance point $\adj{b}$ must be uniformly close to the horoballs associated to any parabolic vertices of $\tilde{T}$ because $G$ acts cocompactly on the truncated hull $\overline{D}$, meaning that $\adj{b}$ is in some uniformly thick part of $\hull{G}$.
Lastly, if $b \in \hull{G}$, then $b \in \overline{D}$ (by the choice of horoballs made during construction; see \cref{def:horoballs}) and therefore $b = \adj{b}$.

It only remains to show that $b$ and $\adj{b}$ are uniformly close even when $b$ lies outside of $\hull{G}$.
Choose a saddle connection $\sigma$ in $T$ for which $\dir{\sigma} \in \partial D$ and $b$ are in different components of the closure of $D - \hull{G}$. 
Denote by $c$ the geodesic from $b$ to $\dir{\sigma}$. 
The length of $\sigma$ decreases along $c$ as it approaches $\dir{\sigma}$. 
The length of $\sigma$ is bounded above at $b$ (\cref{cor:length-of-saddle-connection-bounded-at-balance-point}) and bounded below at a point $t \in c \cap \partial \hull{G}$ (\cref{cor:no-short-saddle-connections-in-truncated-hull}).
Therefore the subsegment of $c$ between $b$ and $t$ has uniformly bounded length.
Since $\adj{b}$ is the closest point projection of $b$ to $\partial \hull{G}$, the distance between $b$ and $\adj{b}$ is no greater than the distance between $b$ and $t$, so $b$ and $\adj{b}$ are uniformly close.
\end{proof}

\begin{corollary}
\label{cor:intersecting-hull-implies-bounded-nonparabolic-saddle-connections}
Let $T = \Delta(x,y,z)$ be a Euclidean triangle consisting of saddle connections $\sigma_x$, $\sigma_y$, and $\sigma_z$ so that $\tilde{T} \cap \hull{G} \neq \emptyset$. 
If $\sigma_x$ is nonparabolic, then it has uniformly bounded length at its associated horopoint. That is, the length of
$$ f_{X_{\dir{\sigma_x}}}(\sigma_x) $$
is uniformly bounded.
\end{corollary}

The next definition comes from {\cite[Section 4.6]{ddls-extensions}}. The following corollary is comparable to {\cite[Corollary 4.11]{ddls-extensions}}, which is given a new proof.

\begin{definition}
Two saddle connections in a common fiber \termttt{span a triangle} if they share an endpoint and the geodesic joining their other endpoints is a single (possibly degenerate) saddle connection. 
Write $\calP(\sigma)$ for the set of saddle connections that span a triangle with $\sigma$. 
\nomenclature[P(sigma)]{$\calP(\sigma)$}{set of saddle connections that span a triangle with $\sigma$}
Denote
$$ B(\sigma) = \bigcup_{\sigma' \in \calP(\sigma)} B_{\dir{\sigma'}} $$
for the union of horoballs or horopoints associated to the saddle connections that span a triangle with $\sigma$. Taking preimages of the horoballs or horopoints, denote 
$$\calB(\sigma) = \bigcup_{\sigma' \in \calP(\sigma)} \calB_{\dir{\sigma'}}.$$
\end{definition}

\begin{corollary}
For any $\sigma' \in \calP(\sigma)$, the sets $B_{\dir{\sigma}}$ and $B_{\dir{\sigma'}}$ are uniformly close in $D$. 
\end{corollary}
\begin{proof}
Suppose saddle connections $\sigma$ and $\sigma'$ span a triangle $T$, and let $\tilde{T}$ be the associated ideal triangle. 
Then $B_{\dir{\sigma}}$ and $B_{\dir{\sigma'}}$ are close by \cref{lem:ideal-triangles-project-to-bounded-diameter-subsets}.
\end{proof}

\subsection{Fans}

\begin{definition}
Let $x,y,z \in \Sigma$. The triangle $\Delta(x,y,z) \subset E_0$ is called a \termttt{fan} if it is nondegenerate and at least two sides consist of a single saddle connection. 
When the information about $x$, $y$, $z$ is clear from context, a fan is sometimes denoted simply by $F$.
\end{definition}
A fan canonically decomposes into a finite union of Euclidean triangles sharing a common vertex. 
See \cref{fig:euclidean-fan} for an illustration and notation used throughout this section.

\begin{definition}
Let $F = \Delta(x,y,z) \subset E_0$ be a fan as in \cref{fig:euclidean-fan}. The \termttt{ideal fan associated to $\Delta(x,y,z)$}, denoted $\Tilde{F}$, is the union of the ideal triangles $\tilde{T}_i \subset D$ associated to the Euclidean triangles $T_i \subset \Delta(x,y,z)$.
\nomenclature[Ti]{$T_i$}{$i$-th triangle in a Euclidean fan}
\nomenclature[Ti tilde]{$\tilde{T}_i$}{$i$-th triangle in an ideal fan, often the ideal triangle associated to the Euclidean triangle $T_i$}
\end{definition}
\cref{fig:structure-lemma} shows the ideal fan associated to the fan from \cref{fig:euclidean-fan}.
For a fan $F$ consisting of $k$ Euclidean triangles, the associated ideal fan $\Tilde{F}$ consists of $k$ ideal triangles with disjoint interiors formed from at most $2k + 1$ vertices; of these vertices, $k-1$ vertices each belong to exactly two ideal triangles.
The following lemma describes that these $k-1$ vertices are contained in a connected component of the boundary at infinity.

\begin{figure}[htbp]
  \centering
  \includegraphics[width=\textwidth]{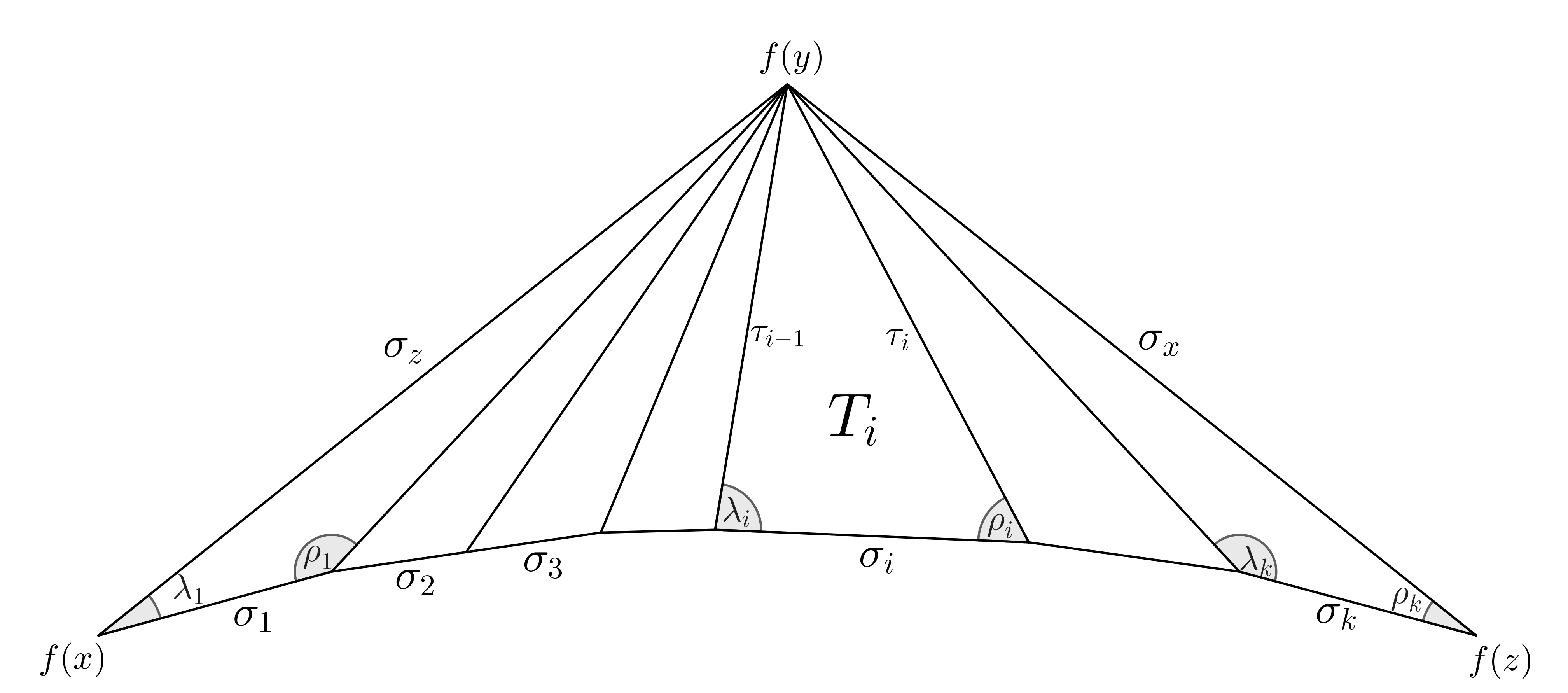}
  \caption{An example of a fan $\Delta(x,y,z)$ in $E_0$ in which $[f(x),f(y)]$ and $[f(y),f(z)]$ are each a single saddle connection, labelled with the notation used in the proof of \cref{lem:structure-of-ideal-fans} and throughout the rest of the paper. In this example, the saddle connections $\sigma_2$ and $\sigma_3$ are parallel, i.e. $\dir{\sigma_2} = \dir{\sigma_3}$.}
  \label{fig:euclidean-fan}
\end{figure}

\begin{lemma}[{\toolttt{Structure lemma}}]
\label{lem:structure-of-ideal-fans}
Let $F = \Delta(x,y,z)$ be a fan consisting of $k$ Euclidean triangles $\set{T_i}_{1 \leq i \leq k}$ as in \cref{fig:euclidean-fan}. 
Then the vertices of the associated ideal fan $\Tilde{F}$ appear in cyclic order (counterclockwise) in $\partial D$: 
$$ \dir{\sigma_z} < \dir{\tau_1} < \dots < \dir{\tau_{k-1}} < \dir{\sigma_x} < \dir{\sigma_k} \leq \dots \leq \dir{\sigma_1} < \dir{\sigma_z}. $$
Moreover, the ideal triangles $\set{\Tilde{T}_i}_{ 1 \leq i \leq k}$ have disjoint interiors. See \cref{fig:structure-lemma}.
\end{lemma}

\begin{figure}[hp]
    \centering
    \includegraphics[width=\textwidth]{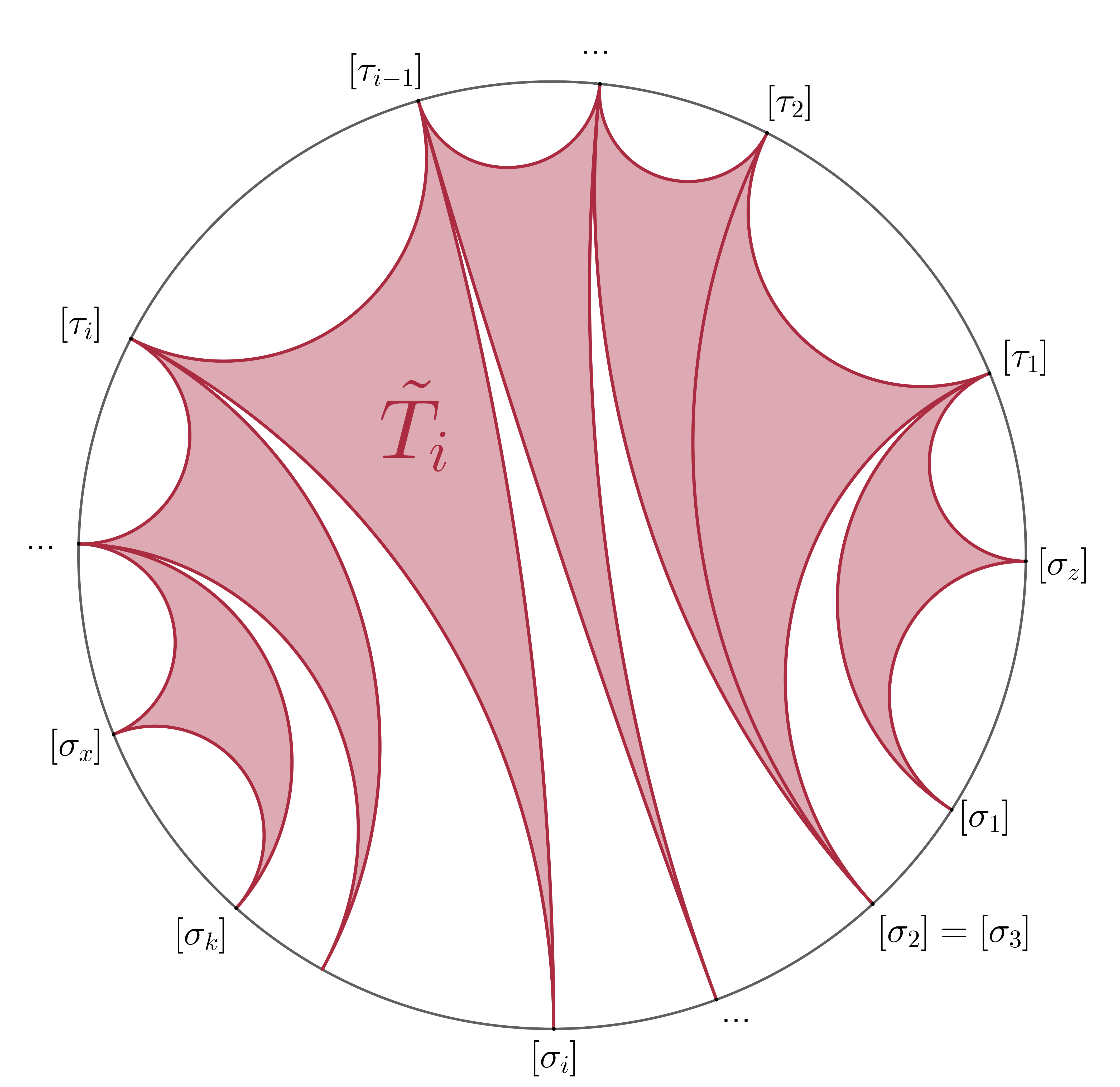}
    \caption{An example of an ideal fan in $D$ (with $\partial D$ represented by the gray circle) corresponding to the Euclidean fan from \cref{fig:euclidean-fan}. As given by the structure lemma (\cref{lem:structure-of-ideal-fans}) each vertex $\dir{\tau_*}$ belongs to exactly two ideal triangles, and $\dir{\sigma_x}$ and $\dir{\sigma_z}$ each belong to exactly one ideal triangle. The remaining vertices may belong to one or more triangles: In this example, the vertex $\dir{\sigma_2} = \dir{\sigma_3}$ belongs to two triangles because $\sigma_2$ and $\sigma_3$ were parallel in the Euclidean fan. Perhaps most importantly, the interiors of the ideal triangles are pairwise disjoint.}
    \label{fig:structure-lemma}
\end{figure}

\begin{proof}
Refer to \cref{fig:euclidean-fan,fig:structure-lemma} for a Euclidean fan, an associated ideal fan, and related notation. 
Develop $F$ into the Euclidean plane; it can be assumed that, after rotating, $\sigma_1$ is in the horizontal direction.
Because $\sigma_1 \sigma_2 \cdots \sigma_k$ is a geodesic in $E_{0}$, $\rho_i + \lambda_{i+1} \geq \pi$, and therefore
$$ \dir{\sigma_k} \leq \dir{\sigma_{k-1}} \leq \cdots \leq \dir{\sigma_1}. $$
Note that for each $T_i$
$$ \dir{\tau_{i-1}} < \dir{\tau_{i}} < \dir{\sigma_i}, $$
and therefore
\begin{align*}
    \dir{\sigma_z} &= \dir{\tau_0} \\
    &< \dir{\tau_{1}} < \dir{\tau_{2}} < \dots < \dir{\tau_{k-2}} < \dir{\tau_{k-1}} < \dir{\tau_k} &&= \dir{\sigma_x}
    \\ &&&< \dir{\sigma_k} \leq \dir{\sigma_{k-1}} \leq \dots \leq \dir{\sigma_{1}} < \dir{\sigma_z}.
\end{align*}

Notice that the inequalities concerning the vertices of $\tilde{T}_1$ are ``nested'' inside those concerning $\tilde{T}_2$---that is, the inequalities $\dir{\sigma_1} < \dir{\sigma_z} < \dir{\tau_{1}} $ relating the vertices of $\tilde{T}_1$ are inserted (as $*$) into the inequalities $\dir{\sigma_2} \leq * < \dir{\tau_{1}} < \dir{\tau_{2}} $ relating the vertices of $\tilde{T}_2$---
the inequalities concerning the vertices of $\tilde{T}_2$ are nested inside those concerning $\tilde{T}_3$, and so on. This demonstrates that the ideal triangles $\Tilde{T}_i$, ${ 1 \leq i \leq k}$ have disjoint interiors and proves the lemma.
\end{proof}

The following corollaries are used in the proofs of the fan lemma (\cref{lem:disjoint-fan-lemma}, \cref{lem:fan-lemma}).

\begin{corollary}
\label{cor:maximal-subfans}
Any ideal fan in $D$ can be split into at most three subfans: $\set{\tilde{T}_1, \dots, \tilde{T}_n}$, which is disjoint from $\hull{G}$; $\set{\tilde{T}_{n+1}, \dots, \tilde{T}_{m-1}}$, for which all triangles intersect $\hull{G}$; and $\set{\tilde{T}_m, \dots, \tilde{T}_k}$, which is also disjoint from $\hull{G}$.
\end{corollary}

Recall that $\DtoDbar: D \rightarrow \overline{D} $ is the $\rho$-closest point projection.

\begin{corollary}
\label{cor:ideal-fan-bounded-projection-to-hull}
Let $F = \Delta(x,y,z)$ be a fan, and let $\tilde{F} \subset D$ be the ideal fan associated to $F$.
If $\tilde{F} \cap \hull{G} = \emptyset$, then there is an upper bound on the diameter of $\DtoDbar(\tilde{F}) \subset \partial \hull{G}$.
\end{corollary}
\begin{proof}
By \cref{lem:structure-of-ideal-fans}, $\DtoDbar(\tilde{F})$ is contained in the image of the projection of one or two ideal triangles in $\tilde{F}$.
In the former case, $\DtoDbar(\tilde{F})$ lies in the image of the projection of a single ideal triangle, which is uniformly bounded (\cref{lem:ideal-triangles-project-to-bounded-diameter-subsets}).
In the latter case these two ideal triangles share a common vertex, so an upper bound on the lengths of the projections of both triangles implies an upper bound on the diameter of $\DtoDbar(\tilde{F})$.
\end{proof}

\subsection{The fan lemma}

This section presents a statement and proof of the fan lemma (\cref{lem:fan-lemma}), which says that the triangle of collapsed preferred paths associated to any Euclidean fan is uniformly slim. It is used later to demonstrate uniform slimness of general triangles of collapsed preferred paths (\cref{thm:collapsed-preferred-paths-form-slim-triangles}).
Slimness of triangles of collapsed preferred paths is the only remaining condition needed to satisfy all hypotheses of the guessing geodesics criterion (\cref{lem:guessing-geodesics}), which gives hyperbolicity of $\hat{E}$ (\cref{thm:hat-e-is-hyperbolic}).

The statement of the fan lemma is identical to that of \cite[Lemma 4.12]{ddls-extensions}, but the proof requires significant reformulation in order to accommodate the circumstance in which the Veech group $G$ is not a lattice.

\begin{lemma}[\toolttt{Fan lemma}]
\label{lem:fan-lemma}
There is $\delta' > 0$ so that if $x,y,z \in \Sigma$ and the geodesic triangle $\Delta(x,y,z)$ in $E_0$ is a fan with $[f(x),f(y)] = \sigma_z$ and $[f(y),f(z)] = \sigma_x$, then the triangle $\Delta^{\hat{\varsigma}}(x,y,z)$ of collapsed preferred paths is $\delta'$-slim. 
Furthermore, if $y$ lies on a geodesic in $D_y$ with endpoints in $\calB(\sigma_z)$ and $\calB(\sigma_x)$,  then the collapsed preferred paths satisfy
$$ \hat{\varsigma}(x,y), \hat{\varsigma}(y,z) \subset N_{\delta'}(\hat{\varsigma}(x,z)) \hspace{1 cm} \text{and} \hspace{1 cm} \hat{\varsigma}(x,z) \subset N_{\delta'}(\hat{\varsigma}(x,y) \cup \hat{\varsigma}(y,z)).$$
\end{lemma}

The ``furthermore'' statement is necessary for getting a uniform slimness constant for general triangles in later arguments.
A Euclidean fan fitting the description in the fan lemma is shown in \cref{fig:euclidean-fan}. The (uncollapsed) preferred paths joining the vertices $x,y,z \in \Sigma$ are described in \cref{tbl:top-of-fan} and \cref{tbl:bottom-of-fan}.

\begin{table}[h]
\begin{center}
\begin{tabular}{ | l l | }
\hline
 $h_z'$ 
    & horizontal geodesic in $D_x$ from $x$ to $X_{\dir{\sigma_z}}$ \Tstrut\Bstrut \\
 $\gamma_z'$ 
    & saddle connection $\sigma_z$ traversed in the fiber over the associated \\
    & horoball or horopoint, i.e. $f_{X_{\dir{\sigma_z}}} \left( \sigma_z \right)$ \Tstrut\Bstrut \\
 $\bar{h}_z'$ 
    & horizontal geodesic in $D_y$ from $X_{\dir{\sigma_z}}$ to $y$ \Tstrut\Bstrut \\  
 $\bar{h}_x'$ 
    & horizontal geodesic in $D_y$ from $y$ to $X_{\dir{\sigma_x}}$ \Tstrut\Bstrut \\
 $\gamma_x'$ 
    & saddle connection $\sigma_x$ traversed in the fiber over the associated \\
    & horoball or horopoint, i.e. $f_{X_{\dir{\sigma_x}}} \left( \sigma_x \right)$ \Tstrut\Bstrut \\
 $h_x'$ 
    & horizontal geodesic in $D_z$ from $X_{\dir{\sigma_x}}$ to $z$ \Tstrut\Bstrut \\
\hline 
\end{tabular}
\end{center}
\caption{In order, the pieces forming the concatenation of the (uncollapsed) preferred paths forming the ``top'' of the fan referred to in \cref{lem:fan-lemma}.
The primed notation on the paths in this table and in \cref{tbl:bottom-of-fan} serves a practical purpose: The unprimed notation is reserved for paths which play a much more active role in the proof of the fan lemma.
}
\label{tbl:top-of-fan}
\end{table}

\begin{table}[h]
\begin{center}
\begin{tabular}{ | l l  |}
\hline
 $h_0'$ 
    & horizontal geodesic in $D_x$ from $x$ to $X_{\dir{\sigma_1}}$ \Tstrut\Bstrut \\
 $\gamma_1'$ 
    & saddle connection $\sigma_1$ traversed in the fiber over the associated\\
    & horoball or horopoint, i.e. $f_{X_{\dir{\sigma_1}}} \left( \sigma_1 \right)$ \Tstrut\Bstrut \\
 $h_1'$ 
    & horizontal geodesic in $D_{\gamma_1'^+} = D_{\gamma_{2}'^-}$ from $X_{\dir{\sigma_1}}$ to $X_{\dir{\sigma_{2}}}$ \Tstrut\Bstrut \\
 \dots 
    & \dots \Tstrut\Bstrut \\
 $\gamma_i'$ 
    & saddle connection $\sigma_i$ traversed in the fiber over the associated\\
    & horoball or horopoint, i.e. $f_{X_{\dir{\sigma_i}}} \left( \sigma_i \right)$ \Tstrut\Bstrut \\
 $h_i'$ 
    & horizontal geodesic in $D_{\gamma_i'^+} = D_{\gamma_{i+1}'^-}$ from $X_{\dir{\sigma_i}}$ to $X_{\dir{\sigma_{i+1}}}$ \Tstrut\Bstrut \\
 \dots 
    & \dots \Tstrut\Bstrut \\
 $\gamma_k'$ 
    & saddle connection $\sigma_k$ traversed in the fiber over the associated\\
    & horoball or horopoint, i.e. $f_{X_{\dir{\sigma_k}}} \left( \sigma_k \right)$ \Tstrut\Bstrut \\
 $h_k'$ 
    & horizontal geodesic in $D_z$ from $X_{\dir{\sigma_k}}$ to $z$ \Tstrut\Bstrut \\
\hline
\end{tabular}
\end{center}
\caption{In order, the pieces forming the (uncollapsed) preferred path forming the ``bottom'' of the fan referred to in \cref{lem:fan-lemma}. 
The primed notation on the paths in this table and in \cref{tbl:top-of-fan} serves a practical purpose: The unprimed notation is reserved for paths which play a much more active role in the proof of the fan lemma. The superscripts $+$ and $-$ denote endpoints of geodesics, as described in \cref{sec:preferred-paths}.}
\label{tbl:bottom-of-fan}
\end{table}

The fan lemma is proved by \cite{ddls-extensions} under the additional hypothesis that $G$ is a lattice. 
In their case, all of the saddle connections are parabolic and therefore have length zero after being projected to $\hat{E}$ by the map $P$. This places most of the burden of the proof on showing that the horizontal pieces of the paths along the top and bottom of the fan are sufficiently close. 
Here $G$ is only required to be finitely generated, so preferred paths may contain nonparabolic saddle connections whose lengths are arbitrarily long even after collapsing. In the proof of the fan lemma that follows, the preferred path along the bottom of the fan is replaced with a substitute path to which it is uniformly close. The substitute path is chosen specifically to allow nonparabolic saddle connections to be traversed in a select set of fibers where slimness of Euclidean fans can be employed. 
The horizontal pieces of the substitute path are handled mostly in the same manner as in the lattice case presented by \cite{ddls-extensions}, except for a few horizontal pieces possibly requiring special treatment.

The following lemma will be necessary on a few occasions.

\begin{lemma}[\toolttt{Tube lemma}; adapted from {\cite[Lemma 4.5]{gm-dehn-filling}}] 
\label{lem:tube-lemma}
Let $\dir{\sigma}, \dir{\sigma'} \in \directions$, and let $v$ be a geodesic in $D$ with endpoints in $\partial B_{\dir{\sigma}}$ and $\partial B_{\dir{\sigma'}}$.
If $w$ is a geodesic in $D$ with endpoints each within a bounded distance of $B_{\dir{\sigma}}$ and $B_{\dir{\sigma'}}$, then $w$ lies within a uniform neighborhood of $v \cup B_{\dir{\sigma}} \cup B_{\dir{\sigma'}}$.
\end{lemma}

To simplify the proof of the fan lemma and to demonstrate the utility of substituting preferred paths, the case in which the ideal fan is disjoint from $\hull{G}$ is presented separately. 

\begin{lemma}[\toolttt{Disjoint fan lemma}]
\label{lem:disjoint-fan-lemma}
The fan lemma (\cref{lem:fan-lemma}) holds under the additional assumption that the ideal fan in $D$ associated to $\Delta(x,y,z)$ does not intersect $\hull{G}$.
\end{lemma}

\begin{proof}
Note that objects are sometimes lifted from $D$ to various horizontal disks. For the sake of simplicity, the lifts are referred to by the same notation wherever it does not cause confusion.

Because the ideal fan never intersects the hull, all of the saddle connections in the preferred paths are traversed in the fibers over their associated horopoints, which lie in a uniformly bounded subsegment of $\partial \hull{G}$ (\cref{cor:ideal-fan-bounded-projection-to-hull}). 
The general trajectory of the proof is to replace the preferred paths with uniformly close substitute paths in which all of the saddle connections are traversed in the fiber over $X_{\dir{\sigma_z}}$, then prove that the triangle of substitute paths is slim.

The substitute path along the ``top'' of the fan is the concatenation of the following paths, in order.
\begin{center}
\begin{tabular}{ | l l }
 $h_z$ 
    & horizontal geodesic in $D_x$ from $x$ to $X_{\dir{\sigma_z}}$ \Tstrut\Bstrut \\ 
 $\gamma_z$ 
    & saddle connection $\sigma_z$ traversed in the fiber over $X_{\dir{\sigma_z}}$, i.e. $f_{X_{\dir{\sigma_z}}} \left( \sigma_z \right)$ \Tstrut\Bstrut \\
 $\bar{h}_z$ 
    & horizontal geodesic in $D_y$ from $X_{\dir{\sigma_z}}$ to $y$ \Tstrut\Bstrut \\  
 $\bar{h}_x$ 
    & horizontal geodesic in $D_y$ from $y$ to $X_{\dir{\sigma_z}}$ \Tstrut\Bstrut \\
 $\gamma_x$ 
    & saddle connection $\sigma_x$ traversed in the fiber over $X_{\dir{\sigma_z}}$, i.e. $f_{X_{\dir{\sigma_z}}} \left( \sigma_x \right)$ \Tstrut\Bstrut \\
 $h_x$ 
    & horizontal geodesic in $D_z$ from $X_{\dir{\sigma_z}}$ to $z$
\end{tabular}
\end{center}
This substitute path is uniformly close to the path in \cref{tbl:top-of-fan}, which can be seen by comparing each pair of analogous pieces as follows.
\begin{itemize}
    \item Since $h_z = h_z'$, these are uniformly close.
    \item Since $\gamma_z = \gamma_z'$, these are uniformly close.
    \item Since $\bar{h}_z = \bar{h}_z'$, these are uniformly close.
    \item Since $\bar{h}_x, \bar{h}_x' \subset D_y$ share an initial point and have terminal points ($X_{\dir{\sigma_z}}$ and $X_{\dir{\sigma_x}}$, respectively) a uniformly bounded distance apart, $\bar{h}_x$ and $\bar{h}_x'$ are uniformly close. 
    \item The saddle connections $\gamma_x$ and $\gamma_x'$ are saddle connections in the fibers over $X_{\dir{\sigma_z}}$ and $X_{\dir{\sigma_x}}$, respectively.
    The map $f_{X_{\dir{\sigma_z}},X_{\dir{\sigma_x}}}$ maps $\gamma_x'$ to $\gamma_x$, and so the distance between them is exactly the distance between $X_{\dir{\sigma_x}}$ and $X_{\dir{\sigma_z}}$, which is uniformly bounded.
    \item Since $h_x, h_x' \subset D_z$ share an initial point and have terminal points ($X_{\dir{\sigma_z}}$ and $X_{\dir{\sigma_x}}$, respectively) a uniformly bounded distance apart, $h_x$ and $h_x'$ are uniformly close. 
\end{itemize}

The substitute path along the ``bottom'' of the fan is the concatenation of the following paths, in order. 
\begin{center}
\begin{tabular}{ | l l }
 $h_0$ 
    & horizontal geodesic in $D_x$ from $x$ to $X_{\dir{\sigma_z}}$ \Tstrut\Bstrut \\
 $\gamma_1$ 
    & saddle connection $\sigma_1$ traversed in the fiber over $X_{\dir{\sigma_z}}$, i.e. $f_{X_{\dir{\sigma_z}}} \left( \sigma_1 \right)$ \Tstrut\Bstrut \\
 $h_1$ 
    & point $\gamma_1^+ = \gamma_2^-$ \Tstrut\Bstrut \\
 \dots 
    & \dots \Tstrut\Bstrut \\
 $\gamma_i$ 
    & saddle connection $\sigma_i$ traversed in the fiber over $X_{\dir{\sigma_z}}$, i.e. $f_{X_{\dir{\sigma_z}}} \left( \sigma_i \right)$ \Tstrut\Bstrut \\
 $h_i$ 
    & point $\gamma_i^+ = \gamma_{i+1}^-$ \Tstrut\Bstrut \\
 \dots 
    & \dots \Tstrut\Bstrut \\
 $\gamma_k$ 
    & saddle connection $\sigma_k$ traversed in the fiber over $X_{\dir{\sigma_z}}$, i.e. $f_{X_{\dir{\sigma_z}}} \left( \sigma_k \right)$ \Tstrut\Bstrut \\
 $h_k$ 
    & horizontal geodesic in $D_z$ from $X_{\dir{\sigma_z}}$ to $z$
\end{tabular}
\end{center}
All but the first and last horizontal paths are points because the concatenation of saddle connections $\gamma_i$ in the fiber over $X_{\dir{\sigma_z}}$ is already continuous.
Again, this substitute path is uniformly close to the path in \cref{tbl:bottom-of-fan}, which can be seen by comparing each pair of analogous pieces as follows.
\begin{itemize}
    \item Since $h_0, h_0' \subset D_x$ share an initial point and have terminal points ($X_{\dir{\sigma_z}}$ and $X_{\dir{\sigma_1}}$, respectively) a uniformly bounded distance apart, $h_0$ and $h_0'$ are uniformly close. 
    \item For all $i$, the saddle connections $\gamma_i$ and $\gamma_i'$ are saddle connections in the fibers over $X_{\dir{\sigma_z}}$ and $X_{\dir{\sigma_i}}$, respectively. 
    The map $f_{X_{\dir{\sigma_z}},X_{\dir{\sigma_i}}}$ maps $\gamma_i'$ to $\gamma_i$, and so the distance between them is exactly the distance between $X_{\dir{\sigma_i}}$ and $X_{\dir{\sigma_z}}$, which is uniformly bounded.
    \item For each $1 \leq i \leq k - 1$, $h_i$ and $h_i'$ have initial points which are the terminal points of $\gamma_i$ and $\gamma_i'$, respectively, and terminal points which are the initial points of $\gamma_{i+1}$ and $\gamma_{i+1}'$, respectively. Since for all $i$ the pair of $\gamma_i, \gamma_i'$ are uniformly close and $h_i$ is contained in a hyperbolic space (with hyperbolicity constant uniform for all $i$), the pair $h_i, h_i'$ must also be uniformly close.
    \item Since $h_k, h_k' \subset D_z$ share a terminal point and have initial points ($X_{\dir{\sigma_z}}$ and $X_{\dir{\sigma_k}}$, respectively) a uniformly bounded distance apart, $h_k$ and $h_k'$ are uniformly close. 
\end{itemize}
Since $h_i$ is degenerate for each $1 \leq i \leq k - 1$, the substitute path along the ``bottom'' of the fan is simply $h_0 \gamma_1 \gamma_2 \cdots \gamma_k h_k$. 
Then because $h_z=h_0$, $\bar{h}_z = \bar{h}_x$ (ignoring orientation), and $h_x = h_k$, to show that the triangle of substitute paths is slim it only remains to show that the saddle connections $\gamma_*$ form a slim triangle.
\begin{claim}
\label{claim:euclidean-fans-are-slim}
The Euclidean fan formed by the saddle connections $\gamma_x,\gamma_z,\gamma_1, \dots, \gamma_k$ in the fiber over $X_{\dir{\sigma_z}}$ is uniformly slim.
\end{claim}
\begin{proof}
Because $X_{\dir{\sigma_z}} \in \partial \hull{G}$ and $G$ acts cocompactly on $\overline{D} \supset \partial \hull{G}$, the space $E_{X_{\dir{\sigma_z}}}$ is uniformly hyperbolic, and therefore the Euclidean fan formed by the saddle connections $\gamma_x,\gamma_z,\gamma_1, \dots, \gamma_k$ in the fiber over $X_{\dir{\sigma_z}}$ is uniformly slim.
\end{proof}

Therefore, the original triangle of preferred paths is slim.

To prove the ``furthermore'' statement, note that since $\bar{h}_z'$ and $\bar{h}_x'$ are both in $D_y$, they form a slim triangle with the geodesic joining their endpoints, denoted $h_y'$.
When $y$ is assumed to lie on a geodesic with endpoints in $\calB(\sigma_z)$ and $\calB(\sigma_x)$, the concatenation of $\bar{h}_z'$ and $\bar{h}_x'$ lies uniformly close to $h_y'$ (\cref{lem:tube-lemma}).
By construction, $h_y'$ is a uniformly bounded subsegment of $\partial \hull{G}$.
Therefore $h_y'$---and the concatenation of $\bar{h}_z'$ and $\bar{h}_x'$---is uniformly close to the shared endpoint of $\gamma_z$ and $\gamma_x$, which results in the containments from the lemma.
\end{proof}

When the ideal fan is not disjoint from $\hull{G}$, the proof is more complex. 
First the fan is decomposed into subfans as in \cref{cor:maximal-subfans}: One subfan consists of all of the ideal triangles that intersect $\hull{G}$, and the other subfans (up to two, if any exist) are disjoint from $\hull{G}$.
Thanks to the tools in \cref{lem:about-balance-points}, the paths associated to the ideal subfan which intersects $\hull{G}$ can be handled in a manner similar to that of \cite{ddls-extensions}, with a few technicalities to consider for the horizontal paths at the beginning and end.
A subfan which is disjoint from $\hull{G}$ is handled in a similar manner to the previous proof; however, some additional work is needed since slimness of the associated Euclidean subfan in any particular fiber over $\partial \hull{G}$ is no longer sufficient to prove slimness of the whole fan.

\begin{proof}[Proof of fan lemma]
Like the disjoint case, note that objects are sometimes lifted from $D$ to other horizontal disks, and the lifts are referred to by the same notation wherever it does not cause confusion.
Finally, note that some arguments are made by showing that the uncollapsed pieces (i.e., before applying $P$) are close, which is sufficient because $P$ is $1$-Lipschitz (\cref{lem:P-is-lipschitz}). 

\hfill

\noindent\textbf{Substituting the concatenated preferred paths along the top of the fan.}
Compare to the concatenated preferred paths described in \cref{tbl:top-of-fan}.
The substitute path along the top of the fan is the concatenation of the following paths, in order.
\begin{center}
\begin{tabular}{ | l l l }
 $h_z$ 
    & $:=h_z'$ 
    & horizontal geodesic in $D_x$ from $x$ to $X_{\dir{\sigma_z}}$ \Tstrut\Bstrut \\ 
 $\gamma_z$ 
    & $:=\gamma_z'$ 
    & saddle connection $\sigma_z$ traversed in the fiber over the associated\\
    && horoball or horopoint, i.e. $f_{X_{\dir{\sigma_z}}} \left( \sigma_z \right)$ \Tstrut\Bstrut \\
 $h_y'$ 
    & 
    & horizontal geodesic in $D_y$ joining $\bar{h}_z^\prime{}^- \in X_{\dir{\sigma_z}}$ to $\bar{h}_x^\prime{}^+ \in X_{\dir{\sigma_x}}$ \Tstrut\Bstrut \\  
 $\gamma_x$ 
    & $:= \gamma_x'$ 
    & saddle connection $\sigma_x$ traversed in the fiber over the associated\\
    && horoball or horopoint, i.e. $f_{X_{\dir{\sigma_x}}} \left( \sigma_x \right)$ \Tstrut\Bstrut \\
 $h_x$ 
    & $:=h_x'$ 
    & horizontal geodesic in $D_z$ from $X_{\dir{\sigma_x}}$ to $z$
\end{tabular}
\end{center}
Since $\bar{h}_z'$ and $\bar{h}_x'$ are both in $D_y$, they form a slim triangle with the geodesic joining their endpoints, denoted $h_y'$.
In particular, $h_y'$ lies in a uniform neighborhood of $\bar{h}_z'$ and $\bar{h}_x'$.
Because the rest of the pieces of the substitute path along the top of the fan are identical to those appearing in the (uncollapsed) preferred paths, 
the substitute path along the top of the fan lies in a uniform neighborhood of $\varsigma(x,y) \cup \varsigma(y,z)$ (while the reverse is not necessarily true).
If in addition $y$ lies on a geodesic with endpoints in $\calB(\sigma_z)$ and $\calB(\sigma_x)$, then the concatenation of $\bar{h}_z'$ with $\bar{h}_x'$ is in fact \emph{uniformly} close to $h_y' \cup B(\sigma_z) \cup B(\sigma_x)$ (\cref{lem:tube-lemma}); therefore the image under $P$ of the substitute path along the top of the fan is uniformly close to $\hat{\varsigma}(x,y) \cup \hat{\varsigma}(y,z)$.

\hfill

\noindent\textbf{Substituting the preferred path along the bottom of the fan.}
Compare to the preferred path described in \cref{tbl:bottom-of-fan}.
There are at most two maximal subfans of the ideal fan which do not intersect $\hull{G}$, and any such subfan must be at the beginning or end of the fan (\cref{cor:maximal-subfans}).
Let $n$ be the largest index for which $\tilde{T}_1, \dots, \tilde{T}_n$ do not intersect $\hull{G}$.
If no such $n$ exists, set $n=0$.
Assume that $n < k$; the case when $n=k$ (that is, when no ideal triangle in the fan intersects $\hull{G}$) is proven in \cref{lem:disjoint-fan-lemma}.
Similarly, let $m$ be the smallest index for which $\tilde{T}_{m}, \dots, \tilde{T}_k$ do not intersect $\hull{G}$.
If no such $m$ exists, set $m=k+1$.
Note that as a result of the structure lemma, $m-n \geq 2$ (\cref{cor:maximal-subfans}).
Define
$$ \gamma_i := \begin{cases}
    f_{X_{\dir{\sigma_z}}}(\sigma_i) 
        & 1 \leq i \leq n \\
    \gamma_i' 
        & n < i < m \\
    f_{X_{\dir{\sigma_x}}}(\sigma_i)
        & m \leq i \leq k
\end{cases}.$$

The horizontal paths are chosen to make the substitute path along the bottom of the fan continuous.
These are explicitly described as follows.
Because the saddle connections associated to the triangles $T_{n+1}, \dots, T_{m-1}$ remain unchanged, the horizontal paths connecting them are also unchanged.
That is, define
$$ h_i := h_i' \hspace{.2in} \text{ for } n+1 \leq i \leq m-2. $$
The paths $\gamma_1 \gamma_2 \cdots \gamma_n$ and $\gamma_m \gamma_{m+1} \cdots \gamma_k$ are the concatenation of the saddle connections $\sigma_1, \dots, \sigma_n$ in the fiber over $X_{\dir{\sigma_z}}$ and the concatenation of the saddle connections $\sigma_m, \dots, \sigma_k$ in the fiber over $X_{\dir{\sigma_x}}$, respectively.
Note in particular that these paths are continuous without the presence of horizontal paths, so the horizontal ``paths'' are chosen to be the endpoints of saddle connections,
$$ h_i := \gamma_i^+ = \gamma_{i+1}^- \hspace{.2in} \text{ for } 1 \leq i \leq n-1 \text{ or } m \leq i \leq k-1.$$
When $n \geq 1$ and $m \leq k$, define the remaining horizontal paths as follows.
\begin{center}
\begin{tabular}{ l l l }
    $h_0$ 
        & $:= h_z$ 
        & horizontal geodesic in $D_x$ from $x$ to $X_{\dir{\sigma_z}}$ \Tstrut\Bstrut \\
    $h_n$ 
        &
        & horizontal geodesic in $D_{\gamma_n^+}$ from $\gamma_n^+$ to $\gamma_{n+1}^-$\Tstrut\Bstrut \\
    $h_{m-1}$ 
        &
        & horizontal geodesic in $D_{\gamma_{m-1}^+}$ from $\gamma_{m-1}^+$ to $\gamma_m^-$ \Tstrut\Bstrut \\
    $h_k$ 
        & $:= h_x$
        & horizontal geodesic in $D_z$ from $X_{\dir{\sigma_x}}$ to $z$
\end{tabular}
\end{center}
If instead $n=0$, define
$ h_0 = h_n := h_0', $
and if instead $m=k+1$, define
$ h_k = h_{m-1} := h_k'. $
In other words, if there is no ideal subfan $\tilde{T}_1, \dots, \tilde{T}_n$ (or $\tilde{T}_m, \dots, \tilde{T}_k$) disjoint from $\hull{G}$, then the beginning (or end) of the substitute path along the bottom of the fan is chosen to be identical to the original preferred path.

The reader will gain the most insight from this proof by assuming in addition that  $n\geq 1$ and $m \leq k$, since this represents the most novel circumstance in which the ideal fan can be split into three subfans depending on their intersection with $\hull{G}$ (see \cref{cor:maximal-subfans}). The following claim, for instance, becomes tautological otherwise.

\begin{claim}
\label{claim:substitute-path-suffices}
The substitute path along the bottom of the fan, $h_0 \gamma_1 h_1 \cdots \gamma_i h_i \cdots \gamma_k h_k$,  is uniformly close to the preferred path along the bottom of the fan, 
    $$\varsigma(x,z) = h_0' \gamma_1' h_1' \cdots \gamma_i' h_i' \cdots \gamma_k' h_k'.$$
\end{claim}
\begin{proof}[Proof of claim.]
If $n=0$ and $m=k+1$---that is, if every triangle in the ideal fan intersects $\hull{G}$---then the preferred path and the substitute path are identical. 
Otherwise if there are triangles in the ideal fan which do not intersect $\hull{G}$, then only the initial and/or terminal portions of the preferred and substitute paths are distinct.
This proof will demonstrate the case when $n \geq 1$; the proof for the case when $m \leq k$ is analogous, and then the case when both $n \geq 1$ and $m \leq k$ follows immediately.

For any $i > n$, $h_i = h_i'$ and $\gamma_i = \gamma_i'$. So it remains to demonstrate that
$$ h_0' \gamma_1' h_1' \cdots \gamma_n' h_n' \text{ and } h_0 \gamma_1 h_1 \cdots \gamma_n h_n $$
are uniformly close. 
This comes as a consequence of \cref{cor:ideal-fan-bounded-projection-to-hull}, which states that the projection of the ideal subfan $\tilde{T}_1, \dots, \tilde{T}_n$ to $\hull{G}$ is a uniformly bounded subsegment of $\partial \hull{G}$.
Therefore the initial segments $h_0'$ and $h_0$ share an endpoint at $x$, and their other endpoints
are uniformly close,
so $h_0'$ and $h_0$ are uniformly close since they lie in a horizontal disk which is uniformly hyperbolic.
Similarly, the final segments $h_n'$ and $h_n$ share an endpoint at $\gamma_{n+1}^-$, and their other endpoints
are uniformly close, 
so $h_n'$ and $h_n$ are uniformly close.
The point $\gamma_i^- = h_{i-1}^+$ is uniformly close to $\gamma_i'^- = h_{i-1}'^+$, and $\gamma_i^+ = h_i^-$ is uniformly close to $\gamma_i'^+ = h_i'^-$, 
so $\gamma_i'$ is uniformly close to $\gamma_i$ (as in the proof of \cref{lem:disjoint-fan-lemma}).
Finally, the endpoints of each $h_i'$, $1 \leq i \leq n-1$, must be within a uniformly bounded distance of each $h_i$, which completes the proof.
\end{proof}

To prove the fan lemma, it now suffices to show that the substitute paths along the top and bottom of the fan are uniformly close to each other.

\hfill

\noindent\textbf{Horizontal pieces at $x$ and $z$ are close.}
When $n \geq 1$, $h_z = h_0$ by construction. 
Otherwise for $n=0$, $h_z$ and $h_0$ share an endpoint at $x$, and their other endpoints are uniformly close because the ideal triangle $\tilde{T}_1$ intersects $\hull{G}$ (\cref{lem:about-balance-points}).
So, in either case, it follows that $h_z$ and $h_0$ are uniformly close.

Analogous arguments show that $h_x$ and $h_k$ are uniformly close. 

\hfill

\noindent\textbf{Saddle connections $\gamma_i$ for $i \leq n$ or $i \geq m$.}
Recall that 
$$ \gamma_1 h_1 \gamma_2 h_2 \cdots h_{n-1} \gamma_n 
    = \gamma_1 \gamma_2 \cdots \gamma_n $$ 
is a continuous path in the fiber over $X_{\dir{\sigma_z}}$.
It forms one side of a Euclidean fan whose other sides are $\gamma_z$ and $\tau_n$---that is, the fan formed by triangles $T_1, \dots, T_n$, which is slim (\cref{claim:euclidean-fans-are-slim}).
By construction, the saddle connection $\tau_n$ is nonparabolic.
Because the ideal triangle $\tilde{T}_{n+1}$ intersects $\hull{G}$, the saddle connection $\tau_n$ has uniformly bounded length in the fiber over $X_{\dir{\tau_{n}}}$ (\cref{cor:intersecting-hull-implies-bounded-nonparabolic-saddle-connections}).
Since $X_{\dir{\tau_{n}}} $ is within a uniformly bounded distance of $X_{\dir{\sigma_z}}$ (\cref{cor:ideal-fan-bounded-projection-to-hull}),
the saddle connection $\tau_n$ has uniformly bounded length in the fiber over $X_{\dir{\sigma_z}}$.
Therefore $\gamma_1 \dots \gamma_n$ is uniformly close to $\gamma_z$ because they are in $E_{X_{\dir{\sigma_z}}}$ which is uniformly hyperbolic.

Analogous arguments show that for $i \geq m$ the saddle connection $\gamma_x$ in the top of the fan is uniformly close to the paths $\gamma_m \gamma_{m+1} \cdots \gamma_k$ in the substitute path along the bottom of the fan.

\hfill

\noindent\textbf{Horizontal pieces $h_i$ for $n+1 \leq i \leq m-2$.} 
The ideal triangles $\tilde{T}_{n+1}$ and $\tilde{T}_{m-1}$ are the first and last ideal triangles intersecting the hull, respectively. 
Define $h$ to be the geodesic in $D$ between the adjusted balance points $\adj{b}_{n+1}$ and $\adj{b}_{m-1}$, and let $h_y$ denote its lift to $D_y$.\footnote{Recall from the beginning of the proof that for the sake of simplicity, objects in $D$ and their lifts to various horizontal disks are generally referred to by the same notation. This is the first (and only) instance when an object in $D$, $h$, is \emph{notationally} distinguished from one of its lifts, $h_y \subset D_y$.}
Because the ideal triangle $\tilde{T}_{n+1}$ intersects $\hull{G}$, there is a uniformly bounded distance from the endpoint $h_y^- = \adj{b}_{n+1}$ to $B_{\dir{\tau_n}}$ (\cref{lem:about-balance-points}), which lies a uniformly bounded distance from $B_{\dir{\sigma_{z}}}$ (indeed if $n=0$, $B_{\dir{\tau_n}}=B_{\dir{\sigma_{z}}}$; otherwise this follows from \cref{cor:ideal-fan-bounded-projection-to-hull}).
Similarly, there is a uniformly bounded distance from the endpoint $h_y^+ = \adj{b}_{m-1}$ to $B_{\dir{\sigma_x}}$.
Because $h_y'^- \in \partial B_{\dir{\sigma_z}}$ and $h_y'^+ \in \partial B_{\dir{\sigma_x}}$, the tube lemma (\cref{lem:tube-lemma}) ensures that $h_y$ lies uniformly close to $h_y' \cup B_{\dir{\sigma_z}} \cup B_{\dir{\sigma_x}}$.
Therefore $P(h_y)$ and $P(h_y')$ must have uniformly bounded Hausdorff distance.

The aim is to show that $P(h_y)$ and $P(h_{n+1} \cup \cdots \cup h_{m-1})$ have uniformly bounded Hausdorff distance. 
This is accomplished by breaking $h_y$ into segments, each of which has uniformly bounded Hausdorff distance to its image under $P$.

Recall that $h_i$ is in the disk $D_{\gamma_i^+}$ connecting the horoballs (or horopoints) $B_{\dir{\sigma_i}}$ and $B_{\dir{\sigma_{i+1}}}$.
Define $h_i''$ to be the geodesic in $D_{\gamma_i^+}$ between (the lifts of) $\adj{b}_i$ and $\adj{b}_{i+1}$, the adjusted balance points of ideal triangles $\tilde{T}_i$ and $\tilde{T}_{i+1}$. 
Because $\tilde{T}_i \cap \hull{G} \neq \emptyset$ and $\tilde{T}_{i+1} \cap \hull{G} \neq \emptyset$, there is a uniform bound on the distance from $\adj{b}_i$ to $B_{\dir{\sigma_i}}$ and from $\adj{b}_{i+1}$ to $B_{\dir{\sigma_{i+1}}}$ (\cref{lem:about-balance-points}).
(Note that if there are multiple saddle connections in the same direction, then $h_i$ might be a point which is uniformly close to both endpoints of $h_i''$.)
Then by the tube lemma (\cref{lem:tube-lemma}), $h_i''$ lies uniformly close (in $D_{\gamma_i^+}$) to $h_i \cup B_{\dir{\sigma_i}} \cup B_{\dir{\sigma_{i+1}}}$.
Therefore $P(h_i'')$ and $P(h_i)$ have uniformly bounded Hausdorff distance.

Define $g_i$ to be the geodesic in $D_y$ joining $\adj{b}_i$ and $\adj{b}_{i+1}$.
Since $\adj{b}_i$ and $\adj{b}_{i+1}$ are both within a uniformly bounded distance of $B_{\dir{\tau_i}}$ (\cref{lem:about-balance-points}),
the saddle connection $\tau_i$ has bounded length over the geodesic $[\adj{b}_i, \adj{b}_{i+1}]$ and therefore $h_i''$ and $g_i$ have bounded Hausdorff distance in $E$ (and in $\overline{E}$). 

The following claim will allow $h_y$ to be broken into segments whose images under $P$ will be shown to be uniformly close to each of the $P(g_i)$.
Recall that $\DbartoDhat : \hull{G} \rightarrow \hat{D}$ is the map that collapses  horoballs in $\overline{D}$ to points.

\begin{claim}
\label{claim:h-close-to-horoballs}
There are points $t_{n+1}, \dots, t_{m-1}$ appearing in order along the geodesic $h$ whose images under $\DbartoDhat: D \rightarrow \hat{D}$ respectively lie within uniformly bounded distance of the collapsed horoballs (or horopoints) $\DbartoDhat(B_{\dir{\sigma_{n+1}}}), \dots, \DbartoDhat(B_{\dir{\sigma_{m-1}}})$.
\end{claim}
This statement is functionally analogous to \cite[Claim 4.15]{ddls-extensions}. Their proof relies on defining $h$ between two balance points, where the associated Euclidean triangles are equilateral, and using the intermediate value theorem to identify points on $h$ which must be close to the balance points of each of the intermediate triangles in the fan. 
The proof below relies instead on the structure lemma (\cref{lem:structure-of-ideal-fans}) to force the same conclusion. Importantly, the approach presented here works in the more general case where $h$ is defined between \emph{adjusted} balance points, where the associated Euclidean triangles might not be equilateral. A key advantage to this construction is that all of the $t_i$ are forced to be in $\hull{G}$, which avoids the issue of later needing to project paths back onto the bundle over the hull.

\begin{proof}[Proof of \cref{claim:h-close-to-horoballs}]
Recall that $h$ is the geodesic joining adjusted balance points $\adj{b}_{n+1}$ and $\adj{b}_{m-1}$.
Recall also that for each $i \in \set{n+1, \dots, m-1}$ the ideal triangle $\tilde{T}_i$ intersects $\hull{G}$ by construction.
Therefore each adjusted balance point $\adj{b}_i$ is within a uniformly bounded distance of all three horoballs or horopoints associated to $\tilde{T}_i$ (\cref{lem:about-balance-points}).
Choose $t_{n+1} := \adj{b}_{n+1}$ and $t_{m-1} := \adj{b}_{m-1}$, which are close to $B_{\dir{\sigma_{n+1}}}$ and $B_{\dir{\sigma_{m-1}}}$, respectively.

By the structure lemma (\cref{lem:structure-of-ideal-fans}), $h$ must pass through all of the ideal triangles $\tilde{T}_i$ for $n+1 \leq i \leq m-1$; and by convexity of $\hull{G}$, $h$ is entirely contained in $\hull{G}$. 
An example is illustrated in \cref{fig:h-close-to-horoballs}.
Therefore if $h$ intersects any of the geodesics joining $\adj{b}_i$ to the horoballs or horopoints associated to $\tilde{T}_i$, then setting $t_i$ to be the intersection point satisfies the claim. (If $h$ intersects more than one of these geodesics, then any choice of intersection point will suffice.)
If $h \cap \tilde{T}_i \subset B_{\dir{\sigma_i}}$, then any choice of $t_i \in h \cap \tilde{T}_i$ will satisfy the claim.

\begin{figure}[H]
    \centering
    \includegraphics[width=\textwidth]{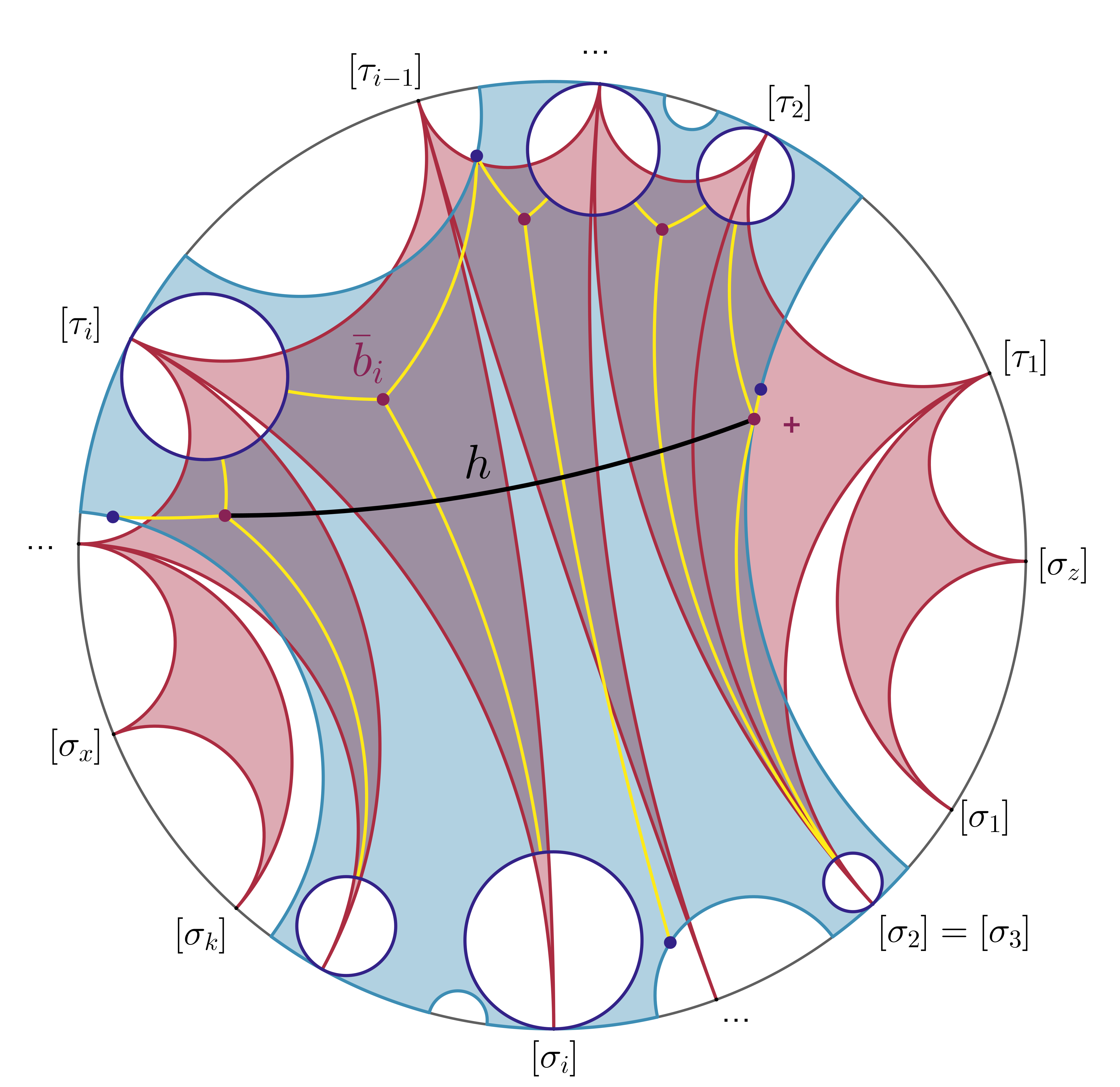}
    \caption{An example of a construction of the $\set{t_i}$ in the proof of \cref{claim:h-close-to-horoballs}. The ideal fan from \cref{fig:structure-lemma} (red) has been superimposed with an example of an approximation to a truncated convex hull $\overline{D}$ (blue). 
    Note that in this drawing, $n=1$ and $m=1$ since there is one triangle disjoint from the hull at the beginning and end of the fan, so the initial and final triangles are not considered by the claim. For the vertices of the fan at parabolic points, the dark blue circles are the boundaries of the horoballs $B_*$. For the vertices of the fan at nonparabolic points, the dark blue dots are the horopoints $B_*$. (Refer to \cref{def:horoballs}.) Then for each ideal triangle the adjusted balance point $\adj{b}_*$ is represented by a red dot; in this example, the only triangle with $\adj{b}_* \neq b_*$ is on the right, where the balance point is represented by a red $+$. Each of the adjusted balance points is a uniformly bounded distance from the horoballs and/or horopoints associated to the same triangle (\cref{lem:about-balance-points}), and the shortest distance to each of these is represented by a yellow geodesic. The geodesic $h$ is drawn in black between adjusted balance points $\adj{b}_{n+1}$ and $\adj{b}_{m-1}$. In this example, all of the $t_*$ are chosen to be the points of intersection of $h$ with the yellow geodesics.}
    \label{fig:h-close-to-horoballs}
\end{figure}

Finally, consider the case in which $h$ avoids intersecting any of the three bounded length geodesics joining $\adj{b}_i$ to the horoballs or horopoints associated to $\tilde{T}_i$ by passing through both of the horoballs $B_{\dir{\tau_{i-1}}}$ and $B_{\dir{\tau_i}}$.
An example is illustrated in \cref{fig:h-close-to-horoballs-technicality}.
By \cref{lem:about-balance-points} and the tube lemma (\cref{lem:tube-lemma}), the point $\adj{b}_i$ lies within a uniform neighborhood of $B_{\dir{\tau_{i-1}}} \cup B_{\dir{\tau_i}}$ and the subsegment of $h$ between these two horoballs,
$$ h_i^\sharp := h \cap \tilde{T}_i \cap \left( B_{\dir{\sigma_i}}^\circ \cup B_{\dir{\tau_{i-1}}}^\circ \cup B_{\dir{\tau_i}}^\circ \right)^C. $$
Choose $t_i$ to be the closest point on $h_i^\sharp$ to $\adj{b}_i$, which may be in $\partial B_{\dir{\tau_{i-1}}}$ or $\partial B_{\dir{\tau_i}}$.
Then $\DbartoDhat(t_i)$ is uniformly close to $\adj{b}_i$, which is uniformly close to $\DbartoDhat(B_{\dir{\sigma_i}})$. 
\begin{figure}[hp]
    \centering
    \includegraphics[width=\textwidth]{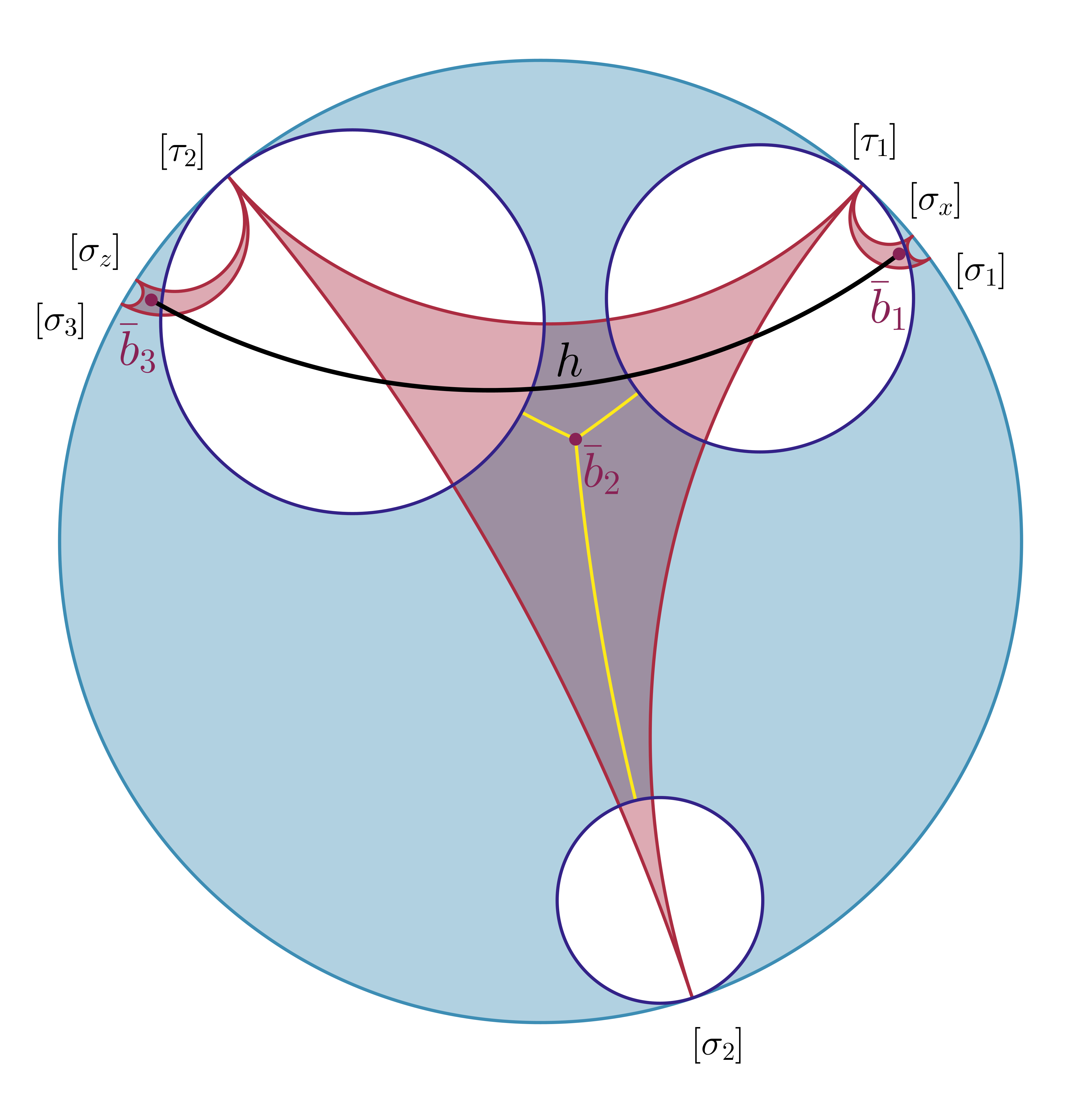}
    \caption{An example of a case in the proof of \cref{claim:h-close-to-horoballs}. The colors and labels are analogous to those in \cref{fig:h-close-to-horoballs}; for simplicity, in this example the convex hull of the limit set of $G$ is assumed to be all of $D$ and some horoballs $B_*$ are not drawn. Because $h$ does not intersect any of the yellow geodesics, $t_2$ will be chosen to be the point on $h$ and in $\tilde{T}_2$ which is closest to $\adj{b}_2$; although $t_2$ and $\adj{b}_2$ may not be uniformly close, their images under $\DbartoDhat$ must be uniformly close. }
    \label{fig:h-close-to-horoballs-technicality}
\end{figure}
\end{proof}

Let $s_i$ be the lift of $t_i$ to $D_y$, and break $h_y$ into segments
$$ h_y^i = [s_i, s_{i+1}], \hspace{.2in} i = n+1, \dots, m-2. $$
By construction, in $D_y$ each $P(s_i)$ is within uniformly bounded distance of $P(B_{\dir{\sigma_i}})$, which is within uniformly bounded distance of $\adj{b}_i$ (note that $P(\adj{b}_i) = \adj{b}_i$).
Recalling that $g_i$ is the geodesic in $D_y$ joining $\adj{b}_i$ and $\adj{b}_{i+1}$, this implies that $P(h_y^i)$ is uniformly close to $P(g_i)$.
Also recalling that $P(g_i)$ and $P(h_i)$ have uniformly bounded Hausdorff distance, each $P(h_i)$ is uniformly close to each $P(h_y^i)$, so $P(h_{n+1} \cup \dots \cup h_{m-2})$ is uniformly close to $P(h_y^{n+1} \cup \dots \cup h_y^{m-2}) = P(h_y)$. 
Finally, recalling that $P(h_y)$ and $P(h_y')$ have uniformly bounded Hausdorff distance, $P(h_{n+1} \cup \cdots \cup h_{m-1})$ and $P(h_y')$ have uniformly bounded Hausdorff distance.
Since $h_y'$ was a piece of the substitute path along the top of the fan and the $h_i$, $n+1 \leq i \leq m-2$ were pieces of the substitute path along the bottom of the fan, this piece of the argument is complete.

Note that in the case where $m-n = 2$---that is, where only one ideal triangle intersects $\hull{G}$---this argument continues to work despite that the geodesic $h$ is in fact not a path but rather the point $\adj{b}_{n+1} = \adj{b}_{m-1}$. In fact, \cref{claim:h-close-to-horoballs} follows immediately from \cref{lem:about-balance-points}.

\hfill

\noindent\textbf{Horizontal pieces $h_n$ ($n \geq 1$) and $h_{m-1}$ ($m \leq k$).}
When $n \geq 1$, the horizontal path $h_n$ still requires consideration. 
Conveniently, $P(h_n)$ must have uniformly bounded length. 
The endpoint $\gamma_n^+$ of $h_n$ is joined to $\gamma_z^+$ by the saddle connection $\tau_n$ in the fiber over $X_{\dir{\sigma_z}}$, where it has uniformly bounded length because the ideal triangle $\tilde{T}_{n+1}$ intersects $\hull{G}$ and because $X_{\dir{\tau_n}}$ is uniformly close to $X_{\dir{\sigma_z}} $(\cref{cor:intersecting-hull-implies-bounded-nonparabolic-saddle-connections}, \cref{cor:ideal-fan-bounded-projection-to-hull}).
Since $\gamma_z^+ = h_y'^-$, the arguments in the previous part of this proof give that $P(\gamma_z^+) = P(h_y'^-)$ is within a uniformly bounded distance of $P(h_y^-) = P(\adj{b}_{n+1})$ in $P(D_y)$, which is within a uniformly bounded distance of the $P$-image of $B_{\dir{\sigma_{n+1}}} \ni \gamma_{n+1}^+ = h_{n+1}^-$ (\cref{claim:h-close-to-horoballs}). 
Lastly, $\gamma_{n+1}$ has uniformly bounded length in the fiber over $X_{\dir{\sigma_{n+1}}}$---either it is a nonparabolic saddle connection associated to an ideal triangle which intersects $\hull{G}$ (\cref{cor:intersecting-hull-implies-bounded-nonparabolic-saddle-connections}), or it is a parabolic saddle connection with bounded length by construction---and $\gamma_{n+1}^- = h_n^+$, so $P(h_n^-)$ and $P(h_n^+)$ must be uniformly close.

When $m \leq k$, a similar argument shows that $P(h_{m-1})$ has uniformly bounded length.

\hfill

\noindent\textbf{Remaining saddle connections.} 
The remaining saddle connections must have uniformly bounded length because their endpoints---which are endpoints of the horizontal paths---have been shown to be uniformly close.
However, it is insightful to see these remaining saddle connections handled explicitly.

Any remaining nonparabolic saddle connections must be associated to ideal triangles which intersect the hull, so each of these saddle connections has uniformly bounded length at its associated horopoint (\cref{cor:intersecting-hull-implies-bounded-nonparabolic-saddle-connections}).
Therefore any isolated saddle connection---that is, any $\sigma_i$ for which $\dir{\sigma_{i-1}} \neq \dir{\sigma_i} \neq \dir{\sigma_{i+1}}$---will not prevent slimness of the collapsed preferred paths. 
(Note that this also resolves any obstruction to slimness posed by the saddle connections $\gamma_z$ and $\gamma_x$ in the cases where $n=0$ and $m=k+1$, respectively.)

Suppose instead that there are consecutive indices 
$$I = \set{i_1, \dots, i_j} \subset \set{n+1, n+2, \dots, m-1}$$ such that $\dir{\sigma_{i_1}} = \dots = \dir{\sigma_{i_j}}$. Denote this common direction by $\dir{\sigma_I}$. Since the ideal triangles $\tilde{T}_{i \in I}$ intersect $\hull{G}$, all the saddle connections in $T_{i \in I}$ have uniformly bounded length at their associated (adjusted) balance points $\adj{b}_i$ (\cref{cor:length-of-saddle-connection-bounded-at-balance-point}), which are all uniformly close to $X_{\dir{\sigma_I}}$ (\cref{lem:about-balance-points}). 
Therefore all of the saddle connections associated to $T_{i \in I}$ have uniformly bounded length in the fiber over $X_{\dir{\sigma_I}}$.
In particular, because all of the $\gamma_{i \in I}$ form one side of the geodesic triangle $f_{X_{\dir{\sigma_I}}}(T_{i_1} \cup \dots \cup T_{i_j})$,
the triangle inequality gives that
$$ \sum_{i \in I} \length{\gamma_{i}} \leq \length{f_{X_{\dir{\sigma_I}}}\left(\tau_{i_1-1}\right)} + \length{f_{X_{\dir{\sigma_I}}}\left(\tau_{i_j} \right)}, $$
which is uniformly bounded above.
Therefore any set of consecutive nonparabolic saddle connections does not prevent slimness of the collapsed preferred paths.
\end{proof}

\subsection{General triangles} 
\label{sec:general-triangles}

It remains to show that any triangle of collapsed preferred paths is uniformly slim.
Equipped with this new version of the fan lemma for finitely generated Veech groups, the subsequent arguments of \cite{ddls-extensions} continue to hold for proving that general triangles of preferred paths are slim. As in \cref{sec:construction}, their results are outlined here and cited wherever their arguments apply essentially verbatim.

\collapsedpreferredpathsformslimtriangles

\begin{proof}[Proof of \cref{thm:collapsed-preferred-paths-form-slim-triangles}]
Let $\delta'$ be the constant from the fan lemma (\cref{lem:fan-lemma}).
Any triangle in $\Delta(x,y,z)$ which has at least one side consisting of exactly one saddle connection can be decomposed into a union of fans; see \cref{fig:decomposed-triangle}. 
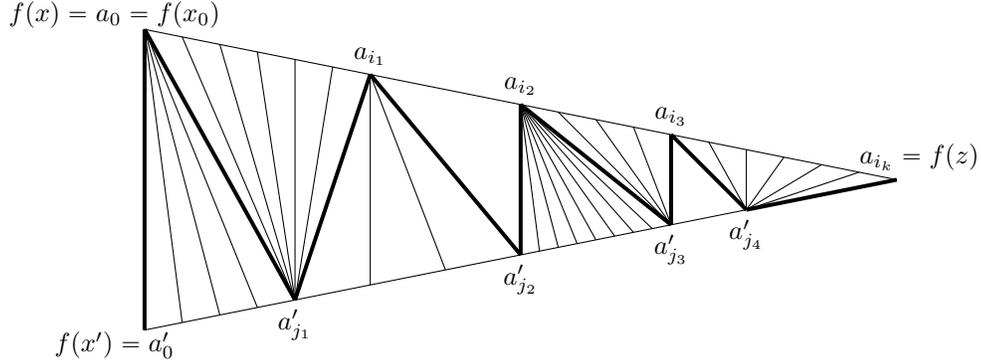
\begin{figure}[htbp]
\begin{center}
\begin{tikzpicture}[scale = 1]
\draw(0,0) -- (-10,-2) -- (-10,2) -- (0,0);
\draw[line width=1.5] (0,0) -- (-2,-.4);
\draw (-2,-.4) -- (-.5,.1);
\draw (-2,-.4) -- (-1,.2);
\draw (-2,-.4) -- (-1.5,.3);
\draw (-2,-.4) -- (-2,.4);
\draw (-2,-.4) -- (-2.5,.5);
\node at (-2,-.4) [below]  {$a'_{j_4}$};
\draw[line width=1.5] (-2,-.4) -- (-3,.6);
\node at (-3, .6) [above] {$a_{i_3}$};
\draw[line width=1.5] (-3,.6) -- (-3,-.6);
\draw (-3,-.6) -- (-3.5,.7);
\draw (-3,-.6) -- (-4,.8);
\draw (-3,-.6) -- (-4.5,.9);
\node at (-3,-.6) [below] {$a'_{j_3}$};
\draw[line width=1.5] (-3,-.6) -- (-5,1);
\draw (-5,1) -- (-3.25,-.65);
\draw (-5,1) -- (-3.5,-.7);
\draw (-5,1) -- (-3.75,-.75);
\draw (-5,1) -- (-4,-.8);
\draw (-5,1) -- (-4.25,-.85);
\draw (-5,1) -- (-4.5,-.9);
\draw (-5,1) -- (-4.75,-.95);
\node at  (-5,1) [above] {$a_{i_2}$};
\draw[line width=1.5] (-5,1) -- (-5,-1);
\node at (-5,-1) [below] {$a'_{j_2}$};
\draw[line width=1.5] (-5,-1) -- (-7,1.4);
\draw (-7,1.4) -- (-6,-1.2);
\draw (-7,1.4) -- (-7,-1.4);
\node at (-7,1.4) [above] {$a_{i_1}$};
\draw[line width=1.5] (-7,1.4) -- (-8,-1.6);
\draw (-8,-1.6) -- (-7.5,1.5);
\draw (-8,-1.6) -- (-8,1.6);
\draw (-8,-1.6) -- (-8.5,1.7);
\draw (-8,-1.6) -- (-9,1.8);
\draw (-8,-1.6) -- (-9.5,1.9);
\node at (-8,-1.6) [below] {$a'_{j_1}$};
\draw[line width=1.5] (-8,-1.6) -- (-10,2);
\draw (-10,2) -- (-9.5,-1.9);
\draw (-10,2) -- (-9,-1.8);
\draw (-10,2) -- (-8.5,-1.7);
\draw[line width=1.5] (-10,2) -- (-10,-2);
\node at (-10.4,2.2)  {$f(x)=a_0 = f(x_0)$};
\node at (-10.4,-2.2) {$f(x')=a'_0$};
\node at (.3,0) [above]  {$a_{i_k} = f(z)$};
\end{tikzpicture}
\caption{Illustration from \cite[Figure 6]{ddls-extensions} of a triangle with exactly one side being a single saddle connection, showing how it can be decomposed into fans. The labeled points are the vertices of the fans.}
\label{fig:decomposed-triangle}
\end{center}
\end{figure}

Slimness of each fan and the ``furthermore'' statement of the fan lemma (\cref{lem:fan-lemma}) ensure that the associated triangle of collapsed preferred paths $\Delta^{\hat{\varsigma}}(x,y,z)$ is $\delta'' = (2 \delta' + 2)$-slim {\cite[Lemma 4.16]{ddls-extensions}}.
General triangles $\Delta(x,y,z)$ may be similarly decomposed into fans, and after a careful treatment of all cases {\cite[Theorem 4.2]{ddls-extensions}} concludes that general triangles of collapsed preferred paths $\Delta^{\hat{\varsigma}}(x,y,z)$ must be $\delta= 3 \delta'' = 3 (2 \delta' + 2)$-slim.
\end{proof}

Coincidentally, this gives an alternate proof for special case of a theorem by \cite{fm-convex-cocompactness} (generalized by \cite{hamenstaedt-hyperbolic-extensions}) which states that a virtually free subgroup of the mapping class group is convex cocompact if and only if its extension group is hyperbolic.

\oldtheoremnewproof

In the sense of \cite{fm-convex-cocompactness}, $G$ is convex cocompact as a discrete subgroup of $\mathrm{Isom}(D)$ if it acts cocompactly on the convex hull of its limit set---that is, if the quotient of $\hull{G}$ by the action of $G$ is compact.
Therefore if $G$ has no parabolic elements, it is convex cocompact as a Fuchsian group.
A consequence is that the $G$-orbit of $X \in \hull{G}$ is quasiconvex.

\begin{proof}
Let $\overline{E}$ be the space constructed in \cref{sec:construction}.
Since there are no horoball preimages to collapse, $\overline{E} = \hat{E}$ and $\Gamma$ acts isometrically and cocompactly on $\overline{E}$ by construction.
It only remains to show that $\overline{E}$ (equipped with the metric $\overline{d}$) is hyperbolic, which is accomplished by applying the guessing geodesics criterion (\cref{lem:guessing-geodesics}).
Here, the sets $L(x,y)$ in the statement of the guessing geodesics criterion are precisely the preferred paths $\varsigma(x,y) = \hat{\varsigma}(x,y)$.

It was shown in the construction that $(\overline{E}, \overline{d})$ is a length space.
Because the collection $\Sigma$ of all cone points is $\Gamma$-invariant and $\overline{E}/\Gamma$ is compact, there exists some constant $R>0$ so that $\Sigma$ is $R$-dense in $\overline{E}$.
For $x,y \in \Sigma$ there is a preferred path $\varsigma(x,y) \subset \overline{E}$, and the preferred paths form slim triangles (\cref{thm:collapsed-preferred-paths-form-slim-triangles}), satisfying condition (1) of the guessing geodesics criterion.

To verify condition (2) of the guessing geodesics criterion, suppose that $\overline{d}(x,y) \leq 3R$ for some $x \in E_X$ and $y \in E_Y$.
Then $\rho(X,Y) \leq 3R$ in $D$.
Because $f_{X,Y} = f_X \vert_{E_Y} $ is $e^{\rho(X,Y)}$-bilipschitz (\cref{sec:total-space-E}), the length of $f_X(\varsigma(x,y))$ is bounded by $3Re^{3R}$.
The geodesic in $E_X$ joining $x$ to $f_X(y)$ also has bounded length and in particular is a concatenation of saddle connections each with length bounded above (by $3Re^{3R}$) and below (uniformly, due to \cref{cor:no-short-saddle-connections-in-truncated-hull}).
So there are at most $n$ saddle connections in $\varsigma(x,y)$, where $n$ depends only on $R$.
Each of these saddle connections $\sigma$ has length bounded above in the fiber over $X$ and bounded below in the fiber at its associated horopoint ${X_{\dir{\sigma}}}$, so $X$ and $X_{\dir{\sigma}}$ are uniformly close for all saddle connections $\sigma$ appearing in $\varsigma(x,y)$.
Finally, since $\varsigma(x,y)$ consists of at most $2n+1$ pieces (hyperbolic geodesics in the horizontal fibers and saddle connections in the vertical fibers), where $n$ depends only on $R$, and each of those pieces has uniformly bounded length, it follows that $L(x,y) = \varsigma(x,y)$ has bounded length---and therefore bounded diameter, as required.
\end{proof}

\section{Hierarchical Hyperbolicity}
\label{sec:hhs}

After establishing a nice action of $\Gamma$ on a hyperbolic space $\hat{E}$, a natural next step is to prove that $\Gamma$ is hierarchically hyperbolic.
\hierarchicalhyperbolicity
Nearly all of \cite{ddls-more-extensions} continues to hold essentially verbatim when $G$ is finitely generated but not necessarily a lattice: Constructing a hierarchically hyperbolic space structure once $\hat{E}$ is shown to be hyperbolic is almost solely concerned with what happens with the Bass-Serre trees $T_\alpha$ for parabolic directions $\alpha$ for $G$, so the existence of nonparabolic directions does not interfere with the construction of the HHS.
For this same reason, $\Gamma$ is already a hierarchically hyperbolic group whenever $G$ has no parabolic limit points; indeed, $\Gamma$ is hyperbolic by \cites{fm-convex-cocompactness}{hamenstaedt-hyperbolic-extensions} (or by \cref{thm:new-old}).
This section is deliberately kept brief; the reader is advised to refer to \cite{ddls-more-extensions} for more comprehensive coverage of the HHS structure and background information.

Let $\calX$ be the flag simplicial complex with $1$-skeleton $\calX^{(1)}$ as given in \cite[Section 4.2]{ddls-more-extensions}. 
Let $\calW$ be the $\calX$-graph whose vertices are the maximal simplices of $\calX$ and whose edges are as given in \cite[Section 4.2]{ddls-more-extensions}.
Let $\calX^{+\calW}$ be the $\calW$-augmented dual graph as given in \cite[Definition 4.2]{ddls-more-extensions}. 
The following statement is comparable to {\cite[Lemma 4.17]{ddls-more-extensions}}.
While this proof is nearly identical to the original, it includes consideration for when nontrivial saddle connections might appear in the case the $G$ is not a lattice.

\begin{lemma}[\toolttt{Empty simplex}]
\label{lem:empty-simplex}
The graph $\calX^{+\calW}$ is quasi-isometric to $\hat{E}$.
\end{lemma}
\begin{proof}
Starting with the map $Z: \calX^{(1)} \rightarrow \sqcup T_\alpha$ defined in \cite[Section 4.2]{ddls-more-extensions}, extend this to a map $Z': \calX^{+\calW} \rightarrow \hat{E}$ as described in the proof of \cite[Lemma 4.17]{ddls-more-extensions}. As in that proof, $Z$ is a one-sided inverse of $Z'$. It remains to show that $Z'$ is coarsely Lipschitz by demonstrating that for any edge $e = [x,y]$ of $\calX^{+\calW}$ with $v=Z(x)$ and $w = Z(y)$ in $\calV$, there exists a path of uniformly bounded length in $\calX^{+\calW}$ joining $v$ and $w$.

Now by \cref{lem:combinatorial-criterion}, any $v \in T_\alpha$, $w \in T_\beta$ are joined by a combinatorial path of length proportional to $\hat{d}(v,w)$.
Since the lengths of the saddle connections are uniformly bounded below by \cref{cor:no-short-saddle-connections-in-truncated-hull} and the combinatorial path is an alternating concatenation of horizontal jumps and saddle connections, the total number of horizontal jumps and saddle connections in the combinatorial path is bounded in terms of $\hat{d}(v,w)$. 
Therefore it suffices to prove the lemma in the cases that $v$ and $w$ are joined by either a horizontal jump or a nonparabolic saddle connection, where either is of uniformly bounded length.
If $v$ and $w$ are joined by a horizontal jump of uniformly bounded length, then the proof proceeds as in \cite{ddls-more-extensions} to produce a path of uniformly bounded length joining $v$ to $w$ in $\calX^{+\calW}$.
Alternatively, if $v$ and $w$ are joined by a saddle connection of uniformly bounded length, then this saddle connection is itself a path of uniformly bounded length joining $v$ to $w$ in $\calX^{+\calW}$.
\end{proof}

\begin{proof}[Proof of \cref{thm:hhs}]
The proof of \cite[Theorem 4.16]{ddls-more-extensions} holds verbatim after substituting \cref{lem:empty-simplex} for \cite[Lemma 4.17]{ddls-more-extensions}.
Therefore the pair $(\calX, \calW)$ is a combinatorial HHS by the definition found in \cite[Definition 4.8]{ddls-more-extensions}.
The rest follows from \cite[Theorem 4.11]{ddls-more-extensions}.
\end{proof}

\pagebreak
\printbibliography[heading=bibintoc,title={Bibliography}]

\pagebreak
\printnomenclature

\end{document}